\documentclass[11pt]{article}

\usepackage[T1]{fontenc}
\usepackage[utf8]{inputenc}
\usepackage{lmodern}
\usepackage[margin=1in]{geometry}
\usepackage{microtype}

\usepackage{amsmath,amssymb,amsthm,mathtools}
\usepackage{array,booktabs,longtable,multirow}
\usepackage[shortlabels]{enumitem}
\usepackage{xcolor}
\usepackage{tikz}
\usetikzlibrary{arrows.meta,positioning,calc,decorations.pathmorphing}
\usepackage{hyperref}
\hypersetup{colorlinks=true,linkcolor=blue!55!black,citecolor=blue!55!black,urlcolor=blue!55!black}
\usepackage[shortlabels]{enumitem}

\usepackage[capitalise]{cleveref}

\newtheorem{theorem}{Theorem}[section]
\newtheorem{proposition}[theorem]{Proposition}
\newtheorem{lemma}[theorem]{Lemma}
\newtheorem{corollary}[theorem]{Corollary}
\newtheorem{definition}[theorem]{Definition}

\theoremstyle{remark}
\newtheorem{remark}[theorem]{Remark}

\crefalias{theorem}{theorem}
\crefalias{definition}{definition}
\crefalias{lemma}{lemma}
\crefalias{proposition}{proposition}
\crefalias{corollary}{corollary}
\crefalias{example}{example}
\crefalias{remark}{remark}

\newcommand{\orb}[1]{[#1]_{\sigma}}

\newcommand{\cF}{\mathcal F}
\newcommand{\cV}{\mathcal V}
\newcommand{\cE}{\mathcal E}

\newcommand{\cX}{\mathcal X}
\newcommand{\Z}{\mathbb Z}
\newcommand{\R}{\mathbb R}

\newcommand{\elltwo}{\ell^2}
\newcommand{\Aut}{\operatorname{Aut}}

\newcommand{\id}{\mathrm{id}}

\usepackage{graphicx}

\title{Regular hyperbolic tilings have no $\ell^2$ eigenfunctions}
\author{Asaf Nachmias\footnote{Department of mathematics, Tel Aviv University,  \texttt{asafnach@tauex.tau.ac.il}}}
\date{}

\begin{document}
\maketitle

\begin{abstract}We show that the adjacency operator of the $1$-skeleton of any regular tiling of the hyperbolic plane has no nonzero square-integrable eigenfunctions. As a consequence, the same holds for every infinite connected regular graph admitting a proper planar embedding with regular dual.
\end{abstract}


\section{Introduction}

The adjacency matrix of any \emph{finite} graph is symmetric and therefore admits an orthonormal basis of eigenfunctions. For an infinite graph, however, the adjacency operator may have no nonzero square-integrable eigenfunctions at all; equivalently, its point spectrum may be empty. Throughout this paper, by an eigenfunction we always mean a square-integrable eigenfunction. Familiar examples with no eigenfunctions include the integer lattices $\Z^d$, for $d\geq 1$, and regular trees.

Furthermore, infinite Cayley graphs can exhibit both point and continuous spectrum. For example, it is not hard to see that the Cayley graph of $\Z\times\Z_2$ with generating set $(1,0), (-1,0), (1,1),(-1,1)$ has both. More importantly, the absence of eigenfunctions is not a group property. Grigorchuk and Pittet \cite[Corollary 1.3]{GrigPittet24} exhibit two finite generating sets of the lamplighter group for which the corresponding adjacency operators have empty point spectrum and pure point spectrum with a basis of finitely supported eigenfunctions.


In this paper we prove the absence of eigenfunctions for the $1$-skeletons of regular tilings of the hyperbolic plane. Recall that for every pair of positive integers $p,q\geq 3$ satisfying ${1\over p}+{1\over q}<{1\over 2}$
there is a tiling of the hyperbolic plane by congruent regular $p$-gons, with precisely $q$ polygons meeting at each vertex. Equivalently, one starts with a regular hyperbolic $p$-gon whose interior angles are $2\pi/q$ and reflects it successively across its sides. The resulting tiling is unique up to isometries of the hyperbolic plane and is called the $(p,q)$ {\bf regular tiling}; see Figure \ref{Fig:37}. The three pairs satisfying
${1\over p}+{1\over q}={1\over 2}$,
namely $(3,6)$, $(4,4)$, and $(6,3)$, give the corresponding regular tilings of the Euclidean plane: the triangular, square, and hexagonal tilings, respectively.
We denote the $(p,q)$ regular tiling by $\cX_{p,q}$. Its $1$-skeleton is the graph whose vertices are the geometric vertices of the tiling, obtained by identifying tile corners that meet at the same point, and whose edges are the sides of the tiles joining adjacent vertices.

\begin{theorem}\label{thm:main} For every pair $(p,q)$ of positive integers satisfying ${1 \over p} + {1 \over q} \leq {1\over 2}$, the adjacency operator of the $1$-skeleton of $\cX_{p,q}$ has no eigenfunctions.  
\end{theorem}


\begin{figure}[t]
\centering
\includegraphics[scale=0.15]{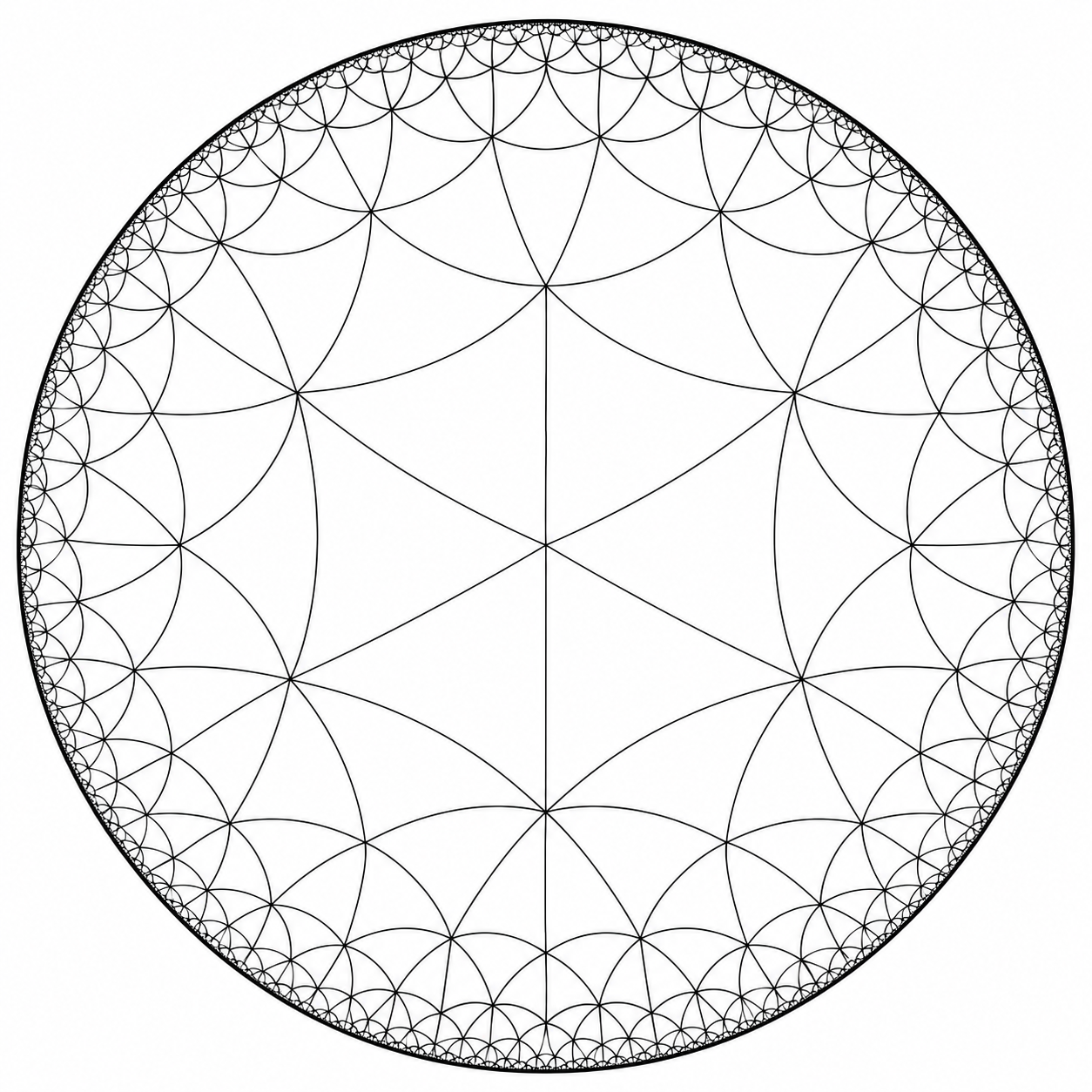}
\caption{\label{Fig:37} The $(3,7)$ regular tiling. Theorem \ref{thm:main} states that for any $\lambda \in \R$ there is no non-zero square-integrable assignment of real numbers to the vertices so that for any vertex the sum of the values of its neighbors equals  $\lambda$ times the value assigned to the vertex itself.}
\end{figure}%

Every infinite connected regular graph with a proper embedding in the plane whose dual is also regular is isomorphic to the $1$-skeleton of a regular $(p,q)$ tiling; see \cite[Remark 4.2]{HaggJonaLyons02}. Theorem \ref{thm:main} therefore immediately yields the following.

\begin{corollary}\label{cor:main}
The adjacency operator of every infinite connected regular graph admitting a proper planar embedding with regular dual has no eigenfunctions. 
\end{corollary}

\subsection{Remarks}

\begin{remark} Previous work has mainly concerned the location of the spectral radius, see \cite{Bar02, DvoMo10, Nag25, Warzel15, Zuk97}. The spectral type has remained much less understood; see \cite[Section 3.2]{Warzel15}. Klassert, Lenz, Peyerimhoff, and Stollmann \cite[Theorem 4]{NoCompact06} proved that regular hyperbolic tilings have no compactly supported eigenfunctions. \cref{thm:main} removes the compact support assumption and proves that the entire point spectrum is empty.
\end{remark}

\begin{remark}
The conclusion of Theorem \ref{thm:main} for the three Euclidean tilings corresponding to ${1 \over p}+{1\over q}={1 \over 2}$ is well known and can easily be shown using the Fourier transform or the methods of this paper.
\end{remark}
\begin{remark} In \cite[Theorem 4]{NoCompact06} it is shown that if ${1 \over p} + {1 \over q} < {1\over 2}$, then the $1$-skeleton of $\cX_{p,q}$ has no \emph{compactly supported} eigenfunctions. Theorem \ref{thm:main} strengthens that result. 
\end{remark}

\begin{remark} Apart from the three Euclidean cases, the conclusion of Theorem \ref{thm:main} is known for some pairs $(p,q)$ with ${1\over p}+{1\over q} < 1/2$. Indeed, let $G$ be a torsion-free group satisfying the strong Atiyah conjecture. Every element $a\in\mathbb C[G]$, viewed as an operator on $\ell^2(G)$ via the left regular representation, has kernel of von Neumann dimension either $0$ or $1$; consequently, every such operator that is not a scalar multiple of the identity has no eigenfunctions, and in particular no Cayley graph of $G$ admits eigenfunctions for its adjacency operator (or Laplacian). Since the strong Atiyah conjecture is known to hold for surface groups \cite{SurfaceAtiyah20}, the conclusion of Theorem \ref{thm:main} follows whenever the $1$-skeleton of $\cX_{p,q}$ is a Cayley graph of the fundamental group of a closed orientable surface. For instance, the $1$-skeleton of the $(8,8)$ regular tiling is the Cayley graph of $\langle a,b,c,d \mid [a,b][c,d]=1 \rangle$, the fundamental group of a genus $2$ surface, with the standard generators. 

This argument, however, applies to only a restricted collection of pairs $(p,q)$. To see this, suppose that the $1$-skeleton of $\cX_{p,q}$ is the Cayley graph of a surface group $G$. The action of $G$ on the Cayley graph is free and transitive on vertices, and the planar cell structure is preserved. The quotient is therefore a closed orientable surface with one vertex. Since every vertex of $\cX_{p,q}$ has degree $q$, the quotient has $q/2$ edges; and since every face has degree $p$, it has $q/p$ faces. In particular, $q$ must be even and $p \mid q$. 
Thus many regular tilings cannot arise in this way. Moreover, many regular tilings are not Cayley graphs of any group. Chaboud and Kenyon \cite{ChaboudKenyon96} showed that the $1$-skeleton of $\cX_{p,q}$ is a Cayley graph if and only if $p$ is divisible by some prime not exceeding $q$. Consequently, for many pairs $(p,q)$ there is no underlying Cayley graph structure at all, and the group-theoretic argument above does not apply. Nonetheless, the $1$-skeleton of a tiling is a Schreier graph, a fact used substantially in this paper.
\end{remark}

\begin{remark} Semiregular tilings of the Euclidean or hyperbolic plane are tilings where more than one type of regular polygons may be used. These may have compactly supported eigenfunctions. See for instance Figure \ref{Fig:31414} for the $(3.14.14)$ semiregular tiling of hyperbolic plane. Choose any face $F$ of $14$ edges and assign values $+1,-1,\ldots,-1$ to its vertices consecutively and $0$ to any other vertex to obtain an eigenvector  of eigenvalue $-2$. 

This shows that the assumption that the dual is regular is essential in  Corollary \ref{cor:main}. Hence, it cannot hold under the assumption that $G$ is an infinite one-ended transitive planar graph. Also, vertex transitive planar graphs that are infinitely ended can also exhibit square-integrable eigenfunctions. Consider for instance the free product of $C_3$ with itself, where $C_3$ is the $3$ element group and take the Cayley graph with the standard generators. This is the $4$-regular graph obtained by replacing each vertex of the $3$-regular tree with a triangle. Choose a base vertex and assign it value $1$, to its $4$ neighbors assign value $-1/2$, to each of their neighbors assign value $1/4$ and so forth. It is straightforward to see that this is an eigenfunction of eigenvalue $-2$. Lastly, the two-ended case can also exhibit compactly supported eigenfunctions, this is the case of the Cayley graph of $\Z \times \Z_2$ mentioned earlier.

\end{remark}

\begin{figure}[t]
\centering
\includegraphics[scale=0.15]{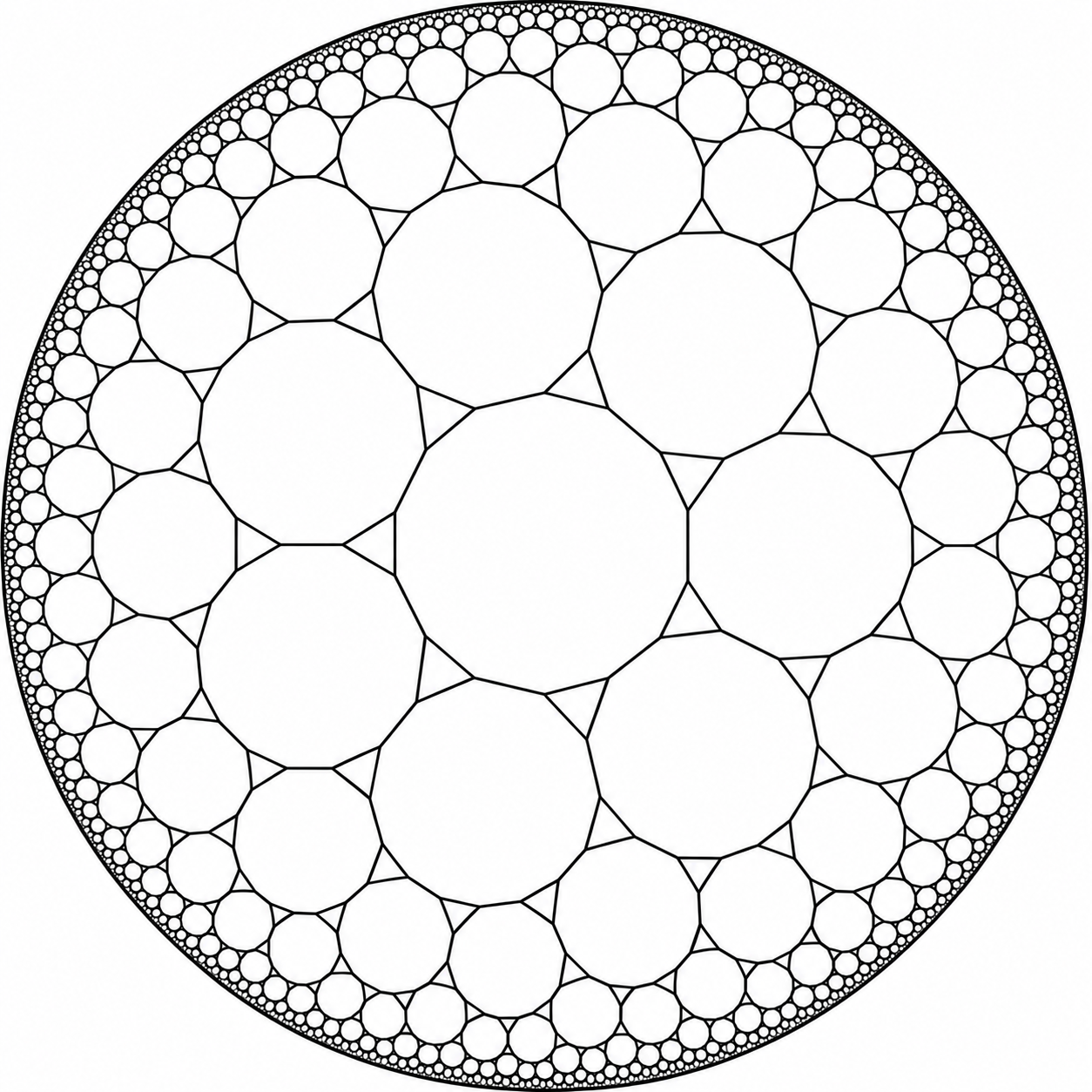}
\caption{\label{Fig:31414} The truncated heptagonal tiling, each vertex belongs to faces of size $3,14$ and $14$.
}
\end{figure}%



\subsection{About the proof}
Our proof relies on a technique of Bordenave, Sen, and Virag \cite{BSV17}. Roughly speaking, if the vertices of a unimodular random graph $(G,o)$ can be invariantly labeled by integers in such a way that most vertices are ``prodigies,'' then one can deduce the absence of eigenfunctions. A vertex is a prodigy if it has a neighbor of lower height, all of whose other neighbors also have height lower than that of the prodigy; see Definition \ref{def:blocks}. We refer to such an integer-valued labeling as a \emph{height function}.


The starting point of this paper is the observation that, for regular tilings, suitable height functions can be constructed by folding the tiling onto a finite map and assigning appropriate values to the edges of the quotient. By \emph{folding}, we mean taking a quotient by a subgroup of the orientation-preserving automorphism group of the tiling which is of finite index and acts cocompactly and freely on the vertices. The resulting finite quotient need not be a surface; more generally, it is an \emph{orbifold}.

Here is a trivial example. Consider $\Z^2$, viewed as the regular $(4,4)$ tiling, and quotient it by its full group of $\Z^2$-translations. The resulting map on the torus has one vertex, one face, and four directed edges, consisting of two edges together with their reverses. Equivalently, this folding can be described directly in terms of edge colors. Color the directed edges of $\Z^2$ by ${0,1,2,3}$ so that, around every clockwise-oriented square, the left, top, right, and bottom edges have colors $0,1,2,3$, respectively. The quotient is obtained by identifying directed edges of the same color, respecting their orientations. On the quotient, assign value $1$ to one directed edge, value $-1$ to its reverse, and value $0$ to the remaining two directed edges. Lifting this assignment to $\Z^2$ gives a $\{-1,0,1\}$-valued function on directed edges. Integrating it along a path from a fixed base vertex to any vertex $v$ yields the usual height function on $\Z^2$.

A nontrivial example is the 
$(5,5)$ regular tiling, see Figure \ref{Fig:55annotated}. Perhaps surprisingly, a very similar construction works. We fold the tiling to a finite quotient with one vertex, one face, and five directed edges. Among these five directed edges there are two pairs consisting of an edge and its reverse, while the remaining edge is its own reverse. To describe the folding, color the directed edges of the $(5,5)$ tiling by ${0,1,2,3,4}$ so that the colors occur around every vertex in the cyclic order $0,1,2,3,4$, and around every face in the cyclic order $0,1,2,4,3$. Such a coloring is shown in Figure \ref{Fig:55annotated}. We then identify directed edges of the same color. Under reversal of orientation, color $0$ is paired with $2$, color $1$ with $3$, while color $4$ is paired with itself. Thus the quotient has one vertex, one face, and five directed edges, with one directed edge identified with its own reverse. Note that the composition of these three permutations on the colors (the face, vertex permutations and edge involution) results in the identity, see Figure \ref{Fig:55annotated}. This is a crucial property that guarantees that the corresponding folding arises from a finite index subgroup of the automorphism group. So in fact we can use many more colors, as long as this property is preserved, and obtain many more foldings (see Definition \ref{def:admissible_permutations} and Theorem \ref{thm:admissible}).

As in the square-lattice example, the boundary of the quotient face contains both a directed edge and its reverse. Assign the values $+1$ and $-1$ to this pair and assign $0$ to the remaining directed edges. The values around the quotient face sum to zero. After lifting this assignment to the $(5,5)$ tiling, we obtain a cocycle: it is orthogonal to every cycle in the tiling, in the sense that its values sum to zero along every cycle. Since the tiling is simply connected, we may  integrate the cocycle from a fixed base vertex to any other vertex obtain a $\Z$-valued equivariant height function, see Lemma \ref{lem:height}. A height function is not enough, we also need to make sure that most vertices are prodigies. Asides from the cocycle condition, in this example it was also important for us that there were no others edges with a $+1$ or $-1$ signs, this guarantees that \emph{all} vertices are prodigy. 

For a general $(p,q)$ tiling, this simple argument is no longer sufficient. In Section \ref{sec:criterion} we formulate conditions on a finite quotient and a cocycle that allow us to construct an equivariant height function for which every vertex is a prodigy. Applying the Bordenave-Sen-Virag Theorem \cite[Theorem 2.3]{BSV17} then yields the absence of eigenfunctions; see Theorem \ref{thm:criterion} and its proof. We also give a combinatorial formulation of these conditions using a permutation representation of the quotient. This leads to a finite certificate, expressed in terms of the defining permutations and cocycle, whose validity can be checked directly; see Definition \ref{def:admissible_permutations} and Theorem \ref{thm:criterion_permutation}. In Section \ref{sec:easypairs} we illustrate the criterion on several ``easy'' tilings for which such certificates can be constructed directly, including the $(5,5)$ regular tiling discussed above.

Lastly, Theorem \ref{thm:admissible} shows that every pair $(p,q)$ admits such a certificate. Its proof relies on a delicate ``sewing'' procedure that produces new certificates from existing ones. By iterating these operations, we reduce the general problem to finitely many pairs $(p,q)$, which serve as ``seeds'' for the construction. Explicit certificates for these seeds are given in Section \ref{app:seeds}. The most difficult pairs tend to be those with small $p$ and $q$ and little symmetry between the roles of vertices and faces. This is reflected in the size of the corresponding certificates: among the seeds, the certificates for $(3,7)$ and $(7,3)$ are the longest.

\subsection{AI statement}\label{sec:ai}
The author used ChatGPT 5.6 Sol in several important parts of the proof. Most significantly, the model designed the ``sewing'' procedure presented in Section \ref{sec:sewing}, including the local modifications of the defining permutations and cocycle and the strategy of combining them so that the proof reduces to finitely many seeds. In addition, the certificates for the $16$ exceptional pairs $(p,q)$ listed in Section \ref{app:seeds} were found by computational searches designed and carried out with the model. Although these certificates were found by computer search, their verification requires no computer assistance and can, albeit tediously, be carried out by hand. The author has verified each of the certificates by hand; the longest took only about ten  unpleasant minutes.

\subsection{Organization}
In the next section we properly formalize the quotient and cocycle condition which together with \cite[Theorem 2.3]{BSV17} yield the absence of eigenfunctions (Theorem \ref{thm:criterion}). There we also define and prove the permutation representation of the quotient. In Section \ref{sec:sewing} we present the sewing procedure, and use it in Section \ref{sec:existenceproof} to prove the existence of certificates for every $(p,q)$. Finally, in Section \ref{app:seeds} we provide the necessary certificate seeds to finish the argument. 

\begin{figure}[t]
\centering
\includegraphics[scale=0.2]{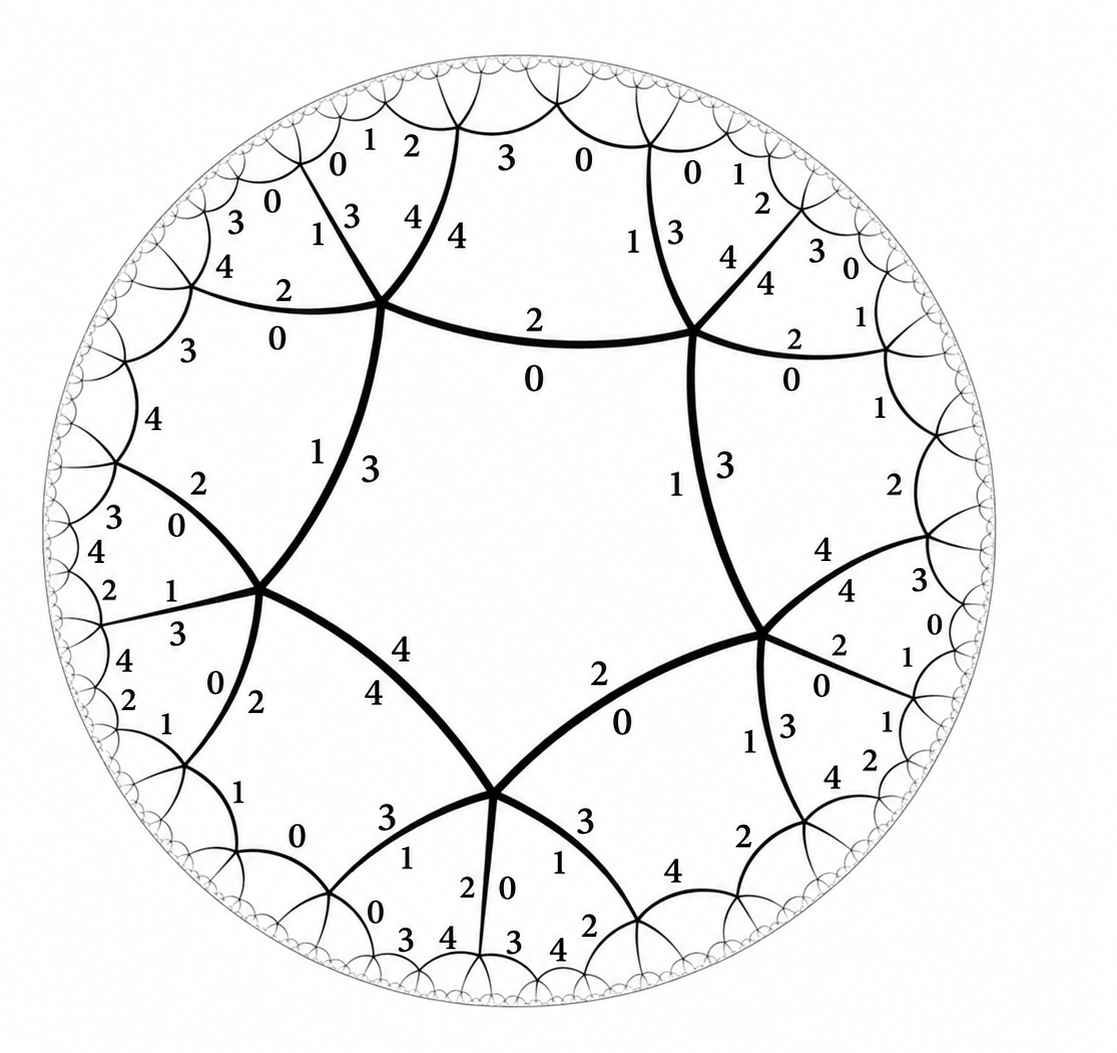}
\caption{\label{Fig:55annotated}The $(5,5)$ regular tiling with colored oriented edges. Each face is oriented clockwise, and the numbers inside it indicate the colors of its oriented boundary edges. The colors occur in the cyclic order $(0\ 1\ 2\ 4\ 3)$ around each face and $(0\ 1\ 2\ 3\ 4)$ around each vertex, while reversal of an edge is described by the involution $(0\ 2)(1\ 3)(4)$. The composition of these three permutations, face permutation followed by vertex permutation followed by edge reversal, is the identity.
 }
\end{figure}%

\section{Quotients, cocycles and permutation representations} \label{sec:criterion}

We regard $\cX_{p,q}$ as a locally finite, simply connected, two-dimensional CW complex whose $0$-cells are the vertices, $1$-cells (edges) are the closed geodesic segments between consecutive vertices in a tile, and the $2$-cells (faces) are the $p$-gons themselves. 
The orientation preserving automorphisms of $\cX_{p,q}$, denoted $\Aut(\cX_{p,q})$ is generated by rotations in angles $\pi$, ${2\pi \over p}$ and ${2\pi \over q}$ around a mid-edge, center of a face and a vertex, respectively. It is isomorphic to the group $\langle a,b,c \mid a^p=b^q=c^2=abc=1 \rangle$.

Let $H\leq\Aut(\mathcal X_{p,q})$ be a subgroup which acts freely and cocompactly on the vertices. Thus, the number of vertex, edge or face orbits is finite and quotient $H \backslash \cX_{p,q}$ is a finite CW complex. We denote the vertices, edges and faces of the quotient by $V(H\backslash \cX_{p,q})$, $E(H \backslash \cX_{p,q})$ and $F(H \backslash \cX_{p,q})$ respectively. The edges and faces inherit the orientation from $\cX_{p,q}$ and we consider the edge and its reverse as distinct elements in $E(H \backslash \cX_{p,q})$ unless they are identified. Note that there may be parallel edges, loops and even edges that are mapped to the reverse image of themselves (if $H$ contains a rotation of angle $\pi$ around the center of an edge) as in Figure \ref{Fig:55annotated}. The faces of $H \backslash \cX_{p,q}$ are also inherited from $\cX_{p,q}$ but may have shorter boundary. For instance, if $H$ contains rotations around the center of a face $F$ so that $F$'s stabilizer has size $r$ with $r \mid p$, then the image of $F$ in the quotient has ${p \over r}$ edges in its boundary. 

\begin{definition}\label{def:cocycle}
An {\bf $H$-cocycle} on $\cX_{p,q}$ is a non-zero function
$u: E(H\backslash \cX_{p,q}) \to \Z$
satisfying:
\begin{enumerate}[label=(\roman*)]
\item \emph{Antisymmetry}. If $e,f\in E(H \backslash \cX_{p,q})$ are the same undirected edge but of reverse orientations, then $u(e)=-u(f)$.
\item \emph{Face law.} If $F \in F(H \backslash \cX_{p,q})$ with oriented boundary $\partial F = e_1+\cdots+e_m$
we have
\[
\sum_{j=1}^m u(e_j)=0 \, .
\]
\end{enumerate}
\end{definition}


Assume $u:E(H\backslash \cX_{p,q}) \to \Z$ is an $H$-cocycle, and set $k=\max _e u(e)$. Antisymmetry implies that $k>0$ and $|u(e)|\le k$ for every $e\in E(H\backslash \cX_{p,q})$. Denote by $v_1,\ldots, v_d$ the vertices $V(H\backslash \cX_{p,q})$  and construct an auxiliary bipartite graph $\mathcal T_u$ on the vertex set  $\{L_1,\ldots,L_d\} \uplus \{R_1,\ldots, R_d\}$ as follows: for each \emph{maximal} edge $e = (v_i,v_j)$ of the quotient, that is, an edge such that $u(e)=k$, add the edge $(L_i, R_j)$ to $\mathcal T_u$. 


\begin{definition}\label{def:monic}
A $H$-cocycle $u:E(H\backslash \cX_{p,q}) \to \Z$ is called {\bf monic} if the bipartite graph $\mathcal T_u$ has a unique perfect matching.
\end{definition}


The following is our criterion of absence of eigenvectors for the $1$-skeleton of CW complexes

\begin{theorem}\label{thm:criterion}
Let $\mathcal X_{p,q}$ be the $(p,q)$ regular tiling and let $H\leq\Aut(\mathcal X_{p,q})$ act freely and cocompactly on the vertices. Let $G$ be the $1$-skeleton of $\mathcal X_{p,q}$. If $\mathcal X_{p,q}$ admits a monic $H$-cocycle, then the adjacency operator of $G$ has no nonzero square-integrable eigenfunction, that is, for every $\lambda\in\R$
\[
\ker(A-\lambda I)=\{0\}
\qquad\text{in }\elltwo(G) \, .
\]
\end{theorem}

\begin{remark}\label{rmk:Bordenave} Using the recent result of Bordenave \cite[Theorem 4]{Bordenave2026} instead of \cite[Theorem 2.3]{BSV17} in the proof of this theorem (see Section \ref{sec:criterionProof}), we may obtain not only that there are no eigenfunctions, but also that the spectral measure of any small interval $I \subset \R$ is at most $O(1/\log(1/|I|))$ where $|I|$ is the length of the interval. We omit the details.    
\end{remark}

It is classical (see \cite{GarethSingerman78}) that quotients such as $H \backslash \cX_{p,q}$ and other maps on surfaces  can be encoded using permutations. In what follows we ``translate'' the conditions of Theorem \ref{thm:criterion}, that is, Definitions \ref{def:cocycle} and \ref{def:monic}, into these notions. This point of view will be particularly useful to us as it will allow us to systematically enlarge certain quotients of $\cX_{p,q}$ to quotients of $\cX_{p',q'}$ for different pairs $(p',q')$ and making sure the new quotients admit a monic cocycle. 

\begin{definition}
\label{def:admissible_permutations} Let $p,q\geq 3$ be integers satisfying ${1 \over p} + {1\over q} \leq 1/2$. A $(p,q)$-{\bf certificate} is a quintuple $(D,\varphi,\sigma,\alpha,u)$ where $D$ is a finite set, $\varphi,\sigma,\alpha$ are permutations of $D$ and $u$ is a non-zero function $u:D\to \Z$ satisfying the following conditions:
\begin{enumerate}
\item The permutations satisfy 
\[
  \alpha^2=1,
  \qquad
  \alpha\sigma\varphi=1 \, ,
\]
and every cycle of $\sigma$ has length $q$, and every cycle of $\varphi$ is of size $p$ or $1$ (in particular $\sigma^q=1$ and $\varphi^p=1$).



\item The group generated by \(\sigma\) and \(\alpha\) acts transitively on \(D\).

\item \(u(\alpha d)=-u(d)\) for every \(d\in D\).
\item For every \(\varphi\)-cycle \(C\),
\[
  \sum_{d\in C}u(d)=0.
\]
\item If \(k=\max_D u>0\), and if \(\mathcal V=D/\langle\sigma\rangle\), then the bipartite multigraph
\[
  \mathcal T_u
\]
with left and right vertex sets both equal to \(\mathcal V\), and with one edge
\[
  L_{[d]}\longrightarrow R_{[\alpha d]}
\]
for every \(d\in D\) satisfying \(u(d)=k\), has a unique perfect matching.

\end{enumerate}
\end{definition}

\begin{theorem} \label{thm:criterion_permutation} Let $(p,q)$ be a pair of positive integers with ${1 \over p}+{1 \over q} \leq {1 \over 2}$. If  there is a $(p,q)$-certificate, then $\cX_{p,q}$ has a monic $H$-cocycle.
\end{theorem}

Finally, the result of this paper is concluded by showing the existence of $(p,q)$-certificates. 

\begin{theorem}\label{thm:admissible}
For every pair of integers $p,q\geq 3$ satisfying ${1 \over p} + {1\over q} \leq 1/2$, there exists a $(p,q)$-certificate.
\end{theorem}

\begin{proof}[Proof of Theorem \ref{thm:main}]
Follows immediately by combining Theorems \ref{thm:criterion},  \ref{thm:criterion_permutation} and \ref{thm:admissible}.
\end{proof}

The proofs of Theorems \ref{thm:criterion} and \ref{thm:criterion_permutation} are presented in Sections \ref{sec:criterionProof} and \ref{sec:criterion_permutation} below. The proof of Theorem \ref{thm:admissible} is more involved and is presented in Sections \ref{sec:sewing}, \ref{sec:existenceproof} and \ref{app:seeds}. But before proceeding, let us observe that Theorems \ref{thm:criterion} and \ref{thm:criterion_permutation} already give us a class of regular tilings with no eigenfunctions. We present them now and later we will use them as ``seeds'' to generate certificates for other pairs $(p,q)$ in order to prove Theorem \ref{thm:admissible}.

\subsection{Some easy certificates}\label{sec:easypairs}

Some hyperbolic tilings have particularly easy certificates.  They correspond to a quotient with one vertex. Let $n\geq 1$ be an integer. In each case $D= \{0,\ldots, n-1\}$ and 
$$ \sigma = (0\ 1\ \cdots\ n-1) \, , \quad \varphi = \sigma^{-1} \alpha \, .$$
The permutations $\alpha$ and the cocycle $u$ are specified in the following; unspecified elements of $D$ are fixed by $\alpha$ and have $u$-value zero. 

\begin{proposition}\label{prop:easypairs} The following are  certificates:
\begin{enumerate}
\item[(E0)] For each $n\geq 4$, a certificate for the pair $(n,n)$ is given by
    $$ \alpha = (0\ 2) (1\ 3) \, , \qquad u(0)=1, \quad u(2)=-1 \, .$$ 

\item[(E1)] For each $n\geq 6$, a certificate for the pair $(n-1,n)$ is given by 
$$ \alpha = (0\ 1) (2\ 4) (3\ n-1), \qquad u(2)=1, \quad u(4)=-1 \, .$$

\item[(E2)] For each $n\geq 8$, a certificate for the pair $(n-2,n)$ is given by
$$\alpha=(0\ 1)(2\ 4)(3\ n-3)(n-2\ n-1)\qquad u(2)=1,\quad u(4)=-1\,.$$
\end{enumerate}
\end{proposition}
\begin{proof}
The proof is just routine verification. In all three cases the quotient has one vertex and a single element of $D$ which receives the maximal value of $u$ so $T_u$ is the graph on two vertices and a single edge. In $(E0)$ the resulting quotient has a single $n$-face, in $(E1)$ there is an $(n-1)$-face and a $1$-face, and lastly in $(E2)$ there is an $(n-2)$-face and two $1$-faces. The rest of the requirements of Definition \ref{def:admissible_permutations} are readily seen to hold.     
\end{proof}

\begin{remark} Note that $(E1)$ does not work for $n=5$ since in that case the $\alpha$ defined is not an involution. For a similar reason $n=6,7$ in $(E2)$ fails. We provide explicit certificates for these $3$ cases in section \ref{app:seeds}
\end{remark}

\subsection{Proof of Theorem \ref{thm:criterion}} \label{sec:criterionProof}

We use the \emph{monotone labeling} technique of Bordenave, Sen and Virag \cite[Theorem 2.3]{BSV17}. Let us recall its setup. Let $G=(V,E)$ be a connected bounded degree infinite graph and $A_G:\elltwo(V)\to\elltwo(V)$ be the adjacency operator. For a vertex $v$, let $\mu_v$ be the spectral measure of $A$ at $\delta_v$, that is,
\[
\int_{\R} f(\lambda)\,d\mu_v(\lambda)
 =\langle \delta_v,f(A)\delta_v\rangle
\]
for bounded Borel functions $f$. The operator $f(A)$ is defined using the Borel functional calculus. In our setup $G$ will be a vertex-transitive graph so $\mu_v$ is the same measure for all vertices $v$. Our goal is to show that $\mu_v(\{\lambda\})=0$ for all $\lambda \in \R$.

\begin{definition}\label{def:blocks}
Let $G=(V,E)$ be a graph. Given a labeling $\eta:V \to \Z$ a vertex
$v$ is called
\begin{enumerate}[label=(\roman*)]
\item {\bf prodigy} if there is a vertex $\hat v$ adjacent to $v$ such that
\[
\eta(\hat v)<\eta(v)
\]
and every other vertex $u\sim\hat v$ with $u\neq v$ satisfies
\[
\eta(u)<\eta(v) \, .
\]
The block $\hat v$ is called a {\bf witness} for $v$ being prodigy.

\item {\bf level} if it is not prodigy and every neighbor of it has label at most $\eta(v)$;
\item {\bf bad} otherwise.
\end{enumerate}
\end{definition}

In the following 
The following is an immediate consequence of \cite[Theorem 2.3]{BSV17}, stated in the form needed here.

\begin{theorem}\label{thm:Bordenave} 
Let $(G,o)$ be a unimodular rooted graph that is almost surely bounded degree, and let $\eta$ be an invariant labeling having $L$ distinct values. Then for any $\lambda\in \R$
\[
\mathbb E\,\mu_o(\{\lambda\})
\le
\mathbb P(o\text{ is level or bad}) \, . \]

\end{theorem}

We proceed with the proof \cref{thm:criterion}. 
Given a subgroup $H\leq \Aut(\cX)$ and an $H$-cocycle $u:E(H\backslash \cX)\to \Z$, let $\tilde u:E(\cX)\to \Z$ be the lift of $u$, i.e., define \(\widetilde u\) on each edge of \(\cX\) to be the value of \(u\) on its corresponding edge in $H\backslash \cX$.

\begin{lemma}\label{lem:height} Let $H\leq \Aut(\cX_{p,q})$ be a subgroup which acts freely and cocompactly on the vertices. If $u:E(H\backslash \cX)\to \Z$ is an $H$-cocycle, then there exists a function
\[
h:V(\mathcal X) \longrightarrow\Z
\]
and a homomorphism
\[
\phi:H\longrightarrow\Z
\]
such that for every oriented edge $(w,v)\in E(\mathcal X)$
\[
h(v)-h(w)=\tilde{u}(w,v)
\]
and for every $g\in H$ and $v\in V(\cX)$ 
\[
h(gv)-h(v)=\phi(g) \, .
\]
\end{lemma}

\begin{proof} By part (ii) of Definition  \ref{def:cocycle}, the sum over the values of $u$ over the edges of a face in $F(H\backslash \cX)$ vanishes. This implies that the sum over the values of $\tilde u$ over the edges of a face in $F\in F(\cX)$ also vanishes. Indeed, if $F$ has $p$ edges then the stabilizer of $F$ in $H$ (containing only rotations of $F$) must be of size $r$ where $r \mid p$; thus each edge of the corresponding quotient edge appears precisely $r$ times in $F$.

Since $\cX$ is simply connected it follows that $\tilde{u}$ vanishes on every cycle of $\cX$, so as usual, we may ``integrate'' over to obtain a primitive $h$. More precisely, for $P=e_1,\ldots,e_m$ forming a directed path in $\cX$ we write
\[
\int_P \tilde{u} :=\sum_{j=1}^m \tilde{u}(e_j) \, .
\]
Fix a base vertex $o$ and for $v\in V(\cX)$, choose a path $P_v$ from $o$ to $v$. Define
\[
h(v) :=\int_{P_v} \tilde{u} \, ,
\]
which is independent of the path $P$ since $\tilde{u}$ vanishes on cycles.
Next, for each $g\in H$, the difference $h(gv)-h(v)$ is independent of $v$. Indeed, integrate along a path from $v$ to $w$ and use $H$-invariance of $\tilde{u}$ to obtain that
\[
h(gw)-h(gv)=h(w)-h(v) \, .
\]
Thus we may define
\[
\phi(g)=h(gv)-h(v) \, ,
\]
and it immediately follows that $\phi$ is a homomorphism. 
\end{proof}

\begin{proof}[Proof of Theorem \ref{thm:criterion}]
Let $\mathcal X$ denote the simply connected complex and $G$ its
$1$-skeleton. Write $Q=H\backslash\mathcal X$, number its vertices
$0,\ldots,d-1$, and let $\tau:V(X)\to \{0,\ldots,d-1\}$ record the quotient vertex of each vertex of $G$. Let $u$ be the monic $H$-cocycle,  $\widetilde u$ be its lift to $G$, and set
$k=\max u$. Since $u$ is nonzero and antisymmetric, $k\geq 1$ and
$|\widetilde u(e)|\leq k$ for every oriented edge $e$.

Recall that $T_u$ has an edge $L_iR_j$ for each oriented quotient edge
from $i$ to $j$ whose $u$-value is $k$, with parallel edges retained. Since $T_u$ has a perfect matching, there exists a permutation $\pi$ of $\{0,\ldots, d-1\}$ so that $L_i R_{\pi(i)}$ is the matching. Since this matching is unique, we can relabel the set $\{0,\ldots, d-1\}$ so that if $L_i R_j$ is a non matching edge of $T_u$, then $j < \pi(i)$. Indeed, this follows since the directed graph on the set $\{0,\ldots, d-1\}$ obtained by adding the edge $j \to \pi(i)$ when $L_iR_j$ is a non matching edge of $T_u$, is acyclic due to the uniqueness of the perfect matching (otherwise we could replace edges of the matching along that cycle and obtain a contradiction to uniqueness).

Apply Lemma \ref{lem:height} and let $h:V(G)\to \Z$ and $\phi:H\to\Z$ denote the corresponding equivariant height functions and homomorphism, respectively.
Define the modified height function $\xi:V(G)\to\Z$ by
\begin{equation*}\label{eq:repaired-modified-height}
  \xi(v)=d\,h(v)+\tau(v).
\end{equation*}
Since $\tau(gv)=\tau(v)$ for any $g\in H$, we have
\begin{equation}\label{eq:repaired-modified-equivariance}
  \xi(gv)=\xi(v)+d\,\phi(g),\qquad g\in H.
\end{equation}

We now observe that every vertex is a prodigy for $\xi$. If $L_i R_{\pi(i)}$ is a matching edge of $T_u$, then it corresponds to a quotient edge $(w,v)$ with $u$-value $k$ so that $\tau(w)=i$ and $\tau(v)=\pi(i)$. By Lemma \ref{lem:height} we have $h(v)-h(w)=k$ and
$$ \xi(v) - \xi(w) = dk + \pi(i) - i \geq dk - (d-1) \geq 1 \, .$$
Let $z \neq v$ be any other neighbor of $w$. If $h(z) - h(w) < k$, since $h$ is $\Z$-valued we get $h(v)-h(z) \geq 1$, hence
$$ \xi(v) - \xi(z) = d(h(v)-h(z)) + \tau(v) - \tau(z) \geq d - (d-1) = 1\, .$$
Otherwise $h(z)-h(w)=k$ and so the edge $(w,z)$ is maximal and present in $T_u$. Since $H$ acts freely on the vertices, the edges $(w,z)$ and $(w,v)$ are different quotient edges. Therefore $\tau(z) < \tau(v)=\pi(i)$. Hence
$$ \xi(v) - \xi(z) = \pi(i) - \tau(z) > 0 \, ,$$
so $v$ is prodigy for $\xi$. 

Next, fix an integer $L>2d(k+1)$, let $\omega$ be drawn uniformly from
$\{0,\ldots,L-1\}$ and define
$$ \eta_\omega(v)=(\xi(v)+\omega)\bmod L \, . $$
The constant $d(k+1)$ is chosen since $|\xi(x)-\xi(y)|\leq d(k+1)$ whenever $x \sim y$. By~\eqref{eq:repaired-modified-equivariance}, for each $g\in H$,
\[
  \eta_\omega(gv)
  =\bigl(\xi(v)+\omega+d\,\phi(g)\bigr)\bmod L.
\]
Adding a fixed residue preserves the uniform law of $\omega$. Hence equivariance guarantees that the law of the entire labeling is $H$-invariant. Furthermore, it is standard to verify unimodularity. Indeed, choose  representatives $v_1,\ldots,v_d$ of the vertex
 orbits, and independently of $\omega$, choose $o$ uniformly from
 these representatives. We verify that $(G,o,\eta_\omega)$ is
 unimodular. Let $F$
 be any nonnegative measurable mass transport, invariant under
 isomorphisms of doubly rooted marked graphs. Freeness implies that
 each vertex has a unique representation $gv_j$ with $g\in H$. Thus
 \begin{align*}
   \mathbb E\sum_{x\in V(G)}F(G, \eta_\omega, o,x)
  &=\frac1d\sum_{i,j=1}^d\sum_{g\in H}
    \mathbb E_\omega F(G, \eta_\omega,v_i,gv_j)\\
  &=\frac1d\sum_{i,j=1}^d\sum_{g\in H}
    \mathbb E_\omega F(G,\eta_\omega,g^{-1}v_i,v_j) =\mathbb E\sum_{x\in V(G)}F(G, \eta_\omega,x,o).
\end{align*}
The second equality uses $H$-invariance of the labeling law; the
last uses the bijection $g\mapsto g^{-1}$ and the orbit
decomposition. 


Let $v$ and its witness $w$ be as above. Since $v$ is prodigy, for any vertex $y\neq v$ that is either $w$ or one of $w$'s neighbors we have that $1 \leq \xi(v)-\xi(y) \leq 2d(k+1)$. If $\eta_\omega(v) = j \geq 2d(k+1)$, then $\eta_\omega(y) = j - (\xi(v)-\xi(y))$ is in $\{0,\ldots, j-1\}$, so $v$ is still prodigy for $\eta_\omega$ with the same witness. It follows that 
$$ \mathbb P(o\text{ is level or bad})\leq\frac{2d(k+1)}{L} \, ,$$
from which Theorem \ref{thm:Bordenave} implies that for any $\lambda\in \R$
we have
$$ \mathbb E\mu_o(\{\lambda\})
  \leq\mathbb P(o\text{ is level or bad})
  \leq\frac{2d(k+1)}{L} \, ,$$
but since our graph is vertex transitive and deterministic, we get that $\mu_v(\{\lambda\}) \leq 2d(k+1)/L$ for any $v\in V$. Taking $L\to\infty$ gives that $\mu_v(\{\lambda\})=0$ for all vertices. We conclude that there are no eigenfunctions. 
\end{proof}


\subsection{Proof of \cref{thm:criterion_permutation}} \label{sec:criterion_permutation}
Let
\[
  \operatorname{Aut}(\cX_{p,q}) =\langle a,b,c\mid a^p=b^q=c^2=abc=1\rangle 
\]
be the orientation-preserving triangle group.  Here \(a\) is the $2\pi/p$ rotation about
the center of a $p$-polygon in $\cX_{p,q}$, \(b\) is the $2\pi/q$ rotation around a vertex, and \(c\) is
the $\pi$ rotation around the midpoint of an edge. These define a right action  $\operatorname{Aut}(\cX_{p,q})$ on the vertices, edges and faces of $\cX_{p,q}$. In particular, on the edges, right multiplication by \(a\) moves one step
around the incident face, right multiplication by \(b\) moves one step around
the initial vertex, and right multiplication by \(c\) reverses the edge.


Next, given a $(p,q)$-certificate we define an action of $\operatorname{Aut}(\cX_{p,q})$ on $D$ by
\[
  d\cdot a=\varphi(d) \, ,
  \qquad
  d\cdot b=\sigma(d) \, ,
  \qquad
  d\cdot c=\alpha(d) \, ,
\]
for each $d\in D$. Indeed, the defining relations of \(\operatorname{Aut}(\cX_{p,q})\) are satisfied due to the first item of Definition \ref{def:admissible_permutations}.  The relation
$d\cdot a^p=d$ holds because each cycle length of $\varphi$ (which are either $1$ or $p$) divides $p$;
the relation $d\cdot b^q=d$ holds because $\sigma$ has cycles only of length $q$; and
$d\cdot c^2=d$ holds because \(\alpha\) is an involution.  Finally, $d\cdot cba = d$, since
$\varphi\sigma\alpha   =\varphi(\varphi^{-1}\alpha)\alpha
  =1$.
  
Let $x_1\in D$ be arbitrary. We define our subgroup $H$ by
\[
  H = H(D,\varphi,\sigma,\alpha) :=\{\gamma\in\operatorname{Aut}(\cX_{p,q}):x_1\cdot\gamma=x_1\} \, ,
\]
that is, $H$ is the stabilizer of $x_1$ of the action above.  

Let us now describe the quotient complex $H\backslash \mathcal X_{p,q}$. Since $\operatorname{Aut}(\cX_{p,q})$ acts freely on the edges of $\cX_{p,q}$, we may choose a base
oriented edge $e_0$ of $\mathcal X_{p,q}$ and identify the edges of $\cX_{p,q}$ with $\operatorname{Aut}(\cX_{p,q})$. The local moves around faces, vertices, and edges are given by right multiplication by
$a,b,c$, respectively. Therefore, $\gamma_1 e_0$ and $\gamma_2 e_0$ correspond to the same edge in $H\backslash \mathcal X_{p,q}$ if and only if $\gamma_1 = h\gamma_2$ for some $h\in H$, hence the left cosets $H\backslash\operatorname{Aut}(\cX_{p,q})$ represent the edges of $H\backslash \mathcal X_{p,q}$. Similarly, the quotient's vertices are identified with $H\backslash \operatorname{Aut}(\cX_{p,q})/\langle  b \rangle$ and the faces are $H\backslash \operatorname{Aut}(\cX_{p,q})/\langle a \rangle$. 
Since $H$ is the stabilizer of $x_1$, the map 
\[
  H\gamma\longmapsto x_1\cdot\gamma
\]
is a bijection from $H\backslash\operatorname{Aut}(\cX_{p,q})$ to $D$ which is identified with the quotient's edges. The group $\operatorname{Aut}(\cX_{p,q})$ acts on the set of left cosets $H \backslash \operatorname{Aut}(\cX_{p,q})$ on the right by $H\gamma \cdot g = H \gamma g$. 
Under the bijection above, this right action of the elements $a,b$ and $c$ become exactly applying $\varphi, \sigma$ and $\alpha$ on $D$. Thus, the quotient's vertices correspond to the $\sigma$-orbits in $D$ and the faces to the $\varphi$-orbits. Furthermore, an element of $\operatorname{Aut}(\cX_{p,q})$ stabilizing a vertex of $\mathcal X_{p,q}$ must be of the form $\gamma  b^n \gamma^{-1}$ for some $\gamma \in \operatorname{Aut}(\cX_{p,q})$ and $n\geq 0$ integer. For such an element to be in $H$ we must have that $x_1 \cdot \gamma b^n = x_1 \cdot \gamma$, but the action of $b$ corresponds to applying $\sigma$. Since $\sigma$ has cycles on of length $q$ it forces $n=0 \mod q$. Hence, $H$ acts freely on the vertices of $\mathcal X_{p,q}$. 

Lastly, the bijection between $H \backslash \operatorname{Aut}(\cX_{p,q})$ to $D$ allows to define $u:E(H\backslash \mathcal X_{p,q})\to \Z$ by the $u:D\to\Z$ provided in Definition \ref{def:admissible_permutations}. This is immediately seen to be a monic $H$-cocycle, due to items $3,4$ and $5$ in Definition \ref{def:admissible_permutations}.
\qed





\section{Vertical,  horizontal and diagonal sewings} \label{sec:sewing}

We focus here on $(p,q)$-certificates defined in Definition \ref{def:admissible_permutations} and show that with additional requirements of them, we may generate new certificates for $(p,q+p)$,  $(p+q,q)$ and $(p+1,q+1)$ we call these operations \emph{vertical, horizontal} and \emph{diagonal sewings}. 

In the vertical sewing we attach a new $p$-face to each vertex and ``sew'' the corresponding cycles together so that every old $q$-vertex becomes a $(q+p)$-vertex; we have to make sure that the sewing does not interfere with the edges that receive maximal $u$-value. The auxiliary graph $T_u$ remains unchanged in this operation. In the horizontal sewing we attach a new $q$-vertex to each $p$-face and we ``sew'' it into the face so that the new faces are all $1$-faces or $(p+q)$-faces. The sewing again must not interfere with edges having maximal $u$-value but more importantly, every new quotient vertex creates a new left vertex and a new right vertex in $T_u$, so we must maintain the uniqueness of the perfect matching in $T_u$. For this reason the horizontal sewing requires more bookkeeping. Lastly in diagonal sewing we enlarge every $p$-face and every vertex by one. 

We will use the following terminology throughout the sewing arguments:

\begin{itemize}[leftmargin=2.2em]
\item An element of $D$ is a {\bf dart}.


\item Elements of the  $\cV=D/\langle\sigma\rangle$ will be called {\bf vertices} and elements of $\cF=D/\langle \varphi \rangle$ are called {\bf faces}. If a face has size $p$ we call it a {\bf $p$-face}. We call $\alpha$-orbits {\bf edges}.

\item A fixed point of $\alpha$ is a {\bf port}.  Antisymmetry forces every
port to have $u$-value $0$.

\item An edge $(a,b)$ is called {\bf charged} if the dart $a$ lies on a $p$-face and the dart $b=\alpha a$ constitutes a $1$-face. In that case $u(a)=u(b)=0$ and $\sigma(b)=a$ so $a$ and $b$ lie in the same vertex.

\item Put $k=\max_Du$.  An edge is called {\bf maximal} if its endpoints take the $u$-values $k$ and $-k$. If the corresponding edge in $T_u$ is a part of the unique perfect matching, we call this edge a {\bf matching edge}. An edge is called {\bf submaximal} if it is not maximal.



\end{itemize}

\subsection{Vertical and horizontal sewing locations}

The following definitions record where the two sewings will be performed.

\begin{definition}\label{def:Vdata}
A {\bf vertical cover} of a certificate $(D,\varphi, \sigma,\alpha, u)$ is
a collection of pairwise disjoint submaximal edges $\{\cE_i\}_{i\in I}$ such that
the sets
\[
       \cV(\cE_i)=\{\orb d:d\in \cE_i\}
\]
form a partition of $\cV$.  Repetitions inside $\cV(\cE_i)$, in the case both darts of the edge belong to the same vertex,  are suppressed. The {\bf support} of a vertical cover are the set of darts contained in $\{\cE_i\}_{i\in I}$. 
\end{definition}

\begin{definition}\label{def:Hdata}
A {\bf horizontal cover} of a certificate $(D,\varphi, \sigma,\alpha, u)$ is a
collection of pairwise disjoint edges $\{\cE_i\}_{i\in I}$ each is one the following:
\begin{enumerate}
\item[(H1)] A matching edge, or
\item[(H2)] A port $z$ on a $p$-face, or
\item[(H3)] A charged dart $c$ on a $p$-face, together with its $1$-face mate
$e=\alpha c$,  
\end{enumerate}
so that the sets 
$$ \cF(\cE_i) = \big \{ [d]_\varphi : d\in \cE_i \, , [d]_\varphi \text{ is a $p$-face} \big \}\, ,$$
form a partition of the set of $p$-faces. We also restrict so that these choices so that in $(H1)$ a matching edge with two endpoints in difference faces only when $q \geq 4$; $(H2)$ only when $q \geq 5$ and $(H3)$ only when $q \geq 7$. The {\bf support} of a horizontal cover are the set of darts contained $\{\cE_i\}_{i\in I}$. 
\end{definition}

\begin{definition}\label{def:bicover} A {\bf bi-cover} of a certificate $(D,\varphi, \sigma,\alpha, u)$ is a vertical cover and a horizontal cover with disjoint supports.   
\end{definition}

\subsection{Vertical sewing}
Vertical sewing enlarges every vertex by $p$ and adds one new $p$-face for
each old vertex.  
In this sewing we add to $D$ new disjoint blocks of size $p$ to every vertex, prescribe the new $\varphi$-cycles to obtain the new $\varphi'$ and keep it fixed. The sewing is then performed by carefully changing $\alpha$ into $\alpha'$. We then set $\sigma'=\alpha' \varphi'^{-1}$.

\begin{definition}\label{def:VSewing} Given a $(p,q)$ certificate $(D,\varphi, \sigma,\alpha, u)$ with a vertical cover $\{\cE_i\}_{i\in I}$, the {\bf vertical sewing} of it is the quintuple $(D',\varphi', \sigma',\alpha', u')$ constructed as follows: \\

\noindent \underline{Step 1}: for each vertex $S\in \cV$ add to $D$ a new block $X(S)$ 
\[
       X(S)=\{x_0, x_1, \ldots, x_{p-1}\} \, , 
\]
and denote this new set of darts by $D'$. We initialize by setting 
$$ \alpha'|_D = \alpha , \quad \alpha'|_{D' \setminus D} = \id, \quad \varphi'|_D = \varphi , \quad \varphi'|_{X(S)} = (x_0\ x_1\ \cdots\  x_{p-1}) \text{ for each } S \in \cV  \, ,$$
and $\sigma' = \alpha' \varphi'^{-1}$.  Furthermore, we set initially
$$ u'|_D = u , \quad u'|_{D' \setminus D} = 0 \, .$$

\noindent \underline{Step 2}: for each $\cE_i$ in the vertical cover, perform one of the following three
operations:

\begin{enumerate}[(V1),leftmargin=3em]
\item If the edge $\cE_i$ is a port $z$ in a vertex $S\in \cV$ 
replace $(z)(x_0)$ in $\alpha'$  by $(z\ x_0)$.

\item If $\cE_i=\{a,b\}$ and $a,b$ belong to the same vertex $S\in \cV$, write $t$ for the value of $u(a)$ (so that $u(b)=-t$ and  $|t|<k$), replace in $\alpha'$
\begin{equation}\label{eq:Vinternal}
       (a\ b)(x_0)(x_1)
       \quad\text{by}\quad
       (a\ x_0)(b\ x_1),
\end{equation}
and set $u'(x_0)=-t$ and $u'(x_1)=t$ (all other $u'$-values remain unaltered).

\item If $\cE_i=\{a,b\}$ joins two distinct vertices $S,T \in \cV$, write $X(S) = \{x_0,\ldots,x_{p-1}\}$ and $X(T)=\{y_0,\ldots, y_{p-1}\}$, replace in $\alpha'$
\begin{equation}\label{eq:Vcross}
 (a\ b)(x_0)(x_1)(y_0)(y_1)
 \quad\text{by}\quad
 (a\ x_0)(b\ y_0)(x_1\ y_1) \, ,
\end{equation}
and if $u(a)=t$, set
\begin{equation}\label{eq:Vcrossu}
 u'(x_0)=-t,\quad u'(y_0)=t,\quad
 u'(x_1)=t,\quad u'(y_1)=-t.
\end{equation}
\end{enumerate}

\noindent \underline{Step 3}: Keep $\varphi'$ fixed and set $\sigma'=\alpha' \varphi'^{-1}$. \\


\end{definition}

\begin{theorem}[Vertical sewing]\label{thm:V} Suppose that $(D,\varphi, \sigma,\alpha, u)$ is a $(p,q)$-certificate which admits a bi-cover. Then the vertical sewing operation defined in Definition \ref{def:VSewing} produces a $(p,q+p)$-certificate $(D',\varphi', \sigma',\alpha', u')$ which admits a bi-cover.
\end{theorem}
\begin{proof} This is just a routine check. In (V1), one transposition joins the old
$\sigma$-cycle of length $q$ to the new $p$-xcycle; the result is a $\sigma'$-cycle of length $q+p$.  Formula~\eqref{eq:Vinternal} does the same in
(V2).  In (V3), we obtain two $\sigma'$-cycles of length $q+p$, one containing
all old darts of $S$ and the other all old darts of $T$.  Since $\{\cV(\cE_i)\}_{i\in I}$ form a partition of $\cV$, the resulting vertices in $D'/ \langle \sigma'\rangle $ all have length exactly $p+q$.

Every new face has sum zero, and the displayed value assignments give
antisymmetry.  All altered old darts and all new nonzero darts are strictly
submaximal.  No maximal edge changes, and each old vertex
is identified with its enlargement.  Hence $T_u$ is
unaltered. 

The edge $\{a,b\}$ used in (V3) is replaced by the edge
$(x_1,y_1)$ which joins the two new corresponding enlarged vertices.  The other two operations do not remove an edge between distinct
vertices. Thus the vertex-adjacency graph remains connected and transitivity
is preserved.

Lastly, since the horizontal and vertical covers of the old certificate have disjoint supports, the horizontal cover's support of the faces of $D$ remains unaltered, so every old $p$-face is still covered by the old horizontal cover. If the vertical sewing was $(V1)$, then the new face $(x_0\ x_1\ x_2\ \cdots\ x_{p-1})$ has two unused ports, namely, $x_1$ and $x_2$; use one for a the horizontal cover of the $p$-face, and one for the vertical cover of the enlarged vertex. If $\cE_i$ was of type  $(V2)$, then the new face has an unused port, namely $x_2$, use it for the horizontal cover, and for the vertical cover, use the edge $\{a, x_0\}$ (or $\{b, x_1\}$), which is submaximal and both sides belongs to the enlarged vertex that the sides of $\cE_i$ belonged to. If $\cE_i$ was of type $(V3)$, then $x_2$ and $y_2$ are ports we use for the horizontal cover of the two new $p$-faces, and since the edge $\{x_1, y_1\}$ is submaximal and covers the enlarged vertices $S$ and $T$, we use it for the vertical cover (the other two edges $\{a,x_0\}$ or $\{b\ y_0\}$ can be used as well). 
\end{proof}

\subsection{Horizontal sewing}

Horizontal sewing enlarges every $p$-face by $q$ and adds one new
$q$-vertex for each old $p$-face.  The construction is necessarily more
elaborate than vertical sewing.  In vertical sewing every new block is
absorbed into an old vertex, so the auxiliary graph $T_u$ does not acquire
new vertices.  In horizontal sewing each new block is itself a new quotient
vertex, and hence contributes left and right vertices to
$T_u$.  The definitions below are designed to preserve the property that the new auxiliary graph $T_u$ has a unique perfect matching. 

For this sewing we add to $D$ new disjoint blocks of size $q$ to every $p$-face, prescribe on them the new vertex cycles to get $\sigma'$ and keep it fixed. Again the actual sewing is performed by changing $\alpha$ to $\alpha'$. Finally we set $\varphi'=\sigma'^{-1}\alpha'$. 

\begin{definition}\label{def:HSewing} Given a $(p,q)$-certificate $(D,\varphi,\sigma,\alpha,u)$ with a
horizontal cover $\{\cE_i\}_{i\in I}$, the {\bf horizontal sewing} of it is the quintuple $(D',\varphi', \sigma',\alpha', u')$ constructed as follows: \\

\noindent \underline{Step 1}:
For every $p$-face $F$, add to $D$ a new block
\[
        X(F)=\{x_0,x_1,\ldots,x_{q-1}\} \, ,
\]
and denote this new set of darts by $D'$. We initially set 
\[
        \sigma'|_D=\sigma,\qquad \alpha'|_D=\alpha,\qquad u'|_D=u \, ,
\]
and set their values on the new blocks as follows:
\begin{enumerate}
    \item If $F$ is horizontally covered by a matching edge, then we initially set   
\[
        \sigma'|_{X(F)}=(x_{q-1}\ x_{q-2}\   \cdots\ x_{1}\ x_0),\qquad \alpha'|_{X(F)}=\id,
        \qquad u'|_{X(F)}=0 \, .
\]

\item If $F$ is horizontally covered by a port or a charged edge (only used when $q \geq 5$), we initially set 
\[
        \alpha'|_X=(x_1\ x_3)(x_2\ x_4),
        \qquad \sigma'|_X=\alpha'|_X (x_{q-1}\  \cdots\ x_{1}\ x_0) \, ,
\]
and 
\[
        u'(x_2)=k,\qquad u'(x_4)=-k,
        \qquad u'(x_i)=0\quad(i\ne2,4).
\]
\end{enumerate}


\medskip
\noindent \underline{Step 2}: For every $\cE_i$ 
perform one of the following
four operations: 

\begin{enumerate}
\item[(H1a)] If $\cE_i=\{a,b\}$ is a matching edge and suppose
that $a$ and $b=\alpha a$ belong to the same $p$-face $F$ and $u(a)=k, u(b)=-k$, then letting $X=X(F)$, replace in
$\alpha'$
\[
        (a\ b)(x_0)(x_1)
        \quad\text{by}\quad
        (a\ x_1)(x_0\ b),
\]
and set
\[
        u'(x_0)=k,\qquad u'(x_1)=-k.
\]

\item[(H1b)] (only used when $q\geq 4)$ If $\cE_i=\{a,b\}$ is a matching edge and suppose
that $a$ and $b$ belong to two distinct $p$-faces $F$ and $G$ and $u(a)=k, u(b)=-k$, then letting
\[
 X(F)=\{x_0,\ldots,x_{q-1}\},\qquad
 X(G)=\{y_0,\ldots,y_{q-1}\} \, ,
\]
replace in $\alpha'$
\[
\begin{aligned}
 (a\ b)(x_0)(x_1)(x_2)(x_3)
       (y_0)(y_1)(y_2)(y_3)
 \quad\text{by}\quad
 (a\ x_0)(b\ x_2)(x_1\ y_0)(x_3\ y_2)(y_1\ y_3),
\end{aligned}
\]
and assign
\[
\begin{array}{c|rrrr}
 &x_0&x_1&x_2&x_3\\ \hline
 u'&-k&-k&k&0
\end{array}
\qquad
\begin{array}{c|rrrr}
 &y_0&y_1&y_2&y_3\\ \hline
 u'&k&-k&0&k.
\end{array}
\]
All other darts of $X(F)\cup X(G)$ retain $u'$-value zero.

\item[(H2)] If $\cE_i=\{z\}$ is a port on a $p$-face $F$, then letting $X=X(F)$, replace in $\alpha'$
\[
        (z)(x_0)\quad\text{by}\quad(z\ x_0).
\]

\item[(H3)] (only used when $q\geq 7$) If $\cE_i=\{c,e\}$ is a charged edge so that $c\in F$ and $e$ is a $1$-face, then letting $X=X(F)$, replace in $\alpha'$
\[
 (c\ e)(x_0)(x_{q-2})(x_{q-1})
 \quad\text{by}\quad
 (c)(e\ x_0)(x_{q-2}\ x_{q-1}).
\]

\end{enumerate}

\noindent \underline{Step 3}: Keep $\sigma'$ fixed and set $\varphi'=\sigma'^{-1}\alpha'$. 
\end{definition}
\medskip 


\begin{theorem}[Horizontal sewing]\label{thm:H}
Suppose that $(D,\varphi,\sigma,\alpha,u)$ is a $(p,q)$-certificate which
admits a bi-cover. Then the
horizontal sewing of Definition~\ref{def:HSewing} produces a
$(p+q,q)$-certificate $(D',\varphi',\sigma',\alpha',u')$
which again admits a bi-cover.  
\end{theorem}

\begin{proof} 

When a face $F$ is horizontally covered by a matching edge, $X(F)$ is ordered by $\sigma'$ cyclically (in Step 1) so it is already a new $q$-vertex. However, at that stage it has no maximal edge touching it. Subsequently in Step 2, we ``sew'' $X(F)$ and $F$ together through the old matching edge covering $F$ and create a new perfect matching. This has to be done differently according to whether the matching edge covers one or two faces. 

Item $(H1a)$ in Step 2 handles the former case; the sewing attaches $X(F)$ into the old face via the matching edge $\{a,b\}$ resulting in a $(p+q)$-face, and the old matching edge $L_{\orb a} R_{\orb b}$ in $T_u$ is replaced in $T_{u'}$ by the two edges $L_{\orb a} R_X$ and  $L_X R_{\orb b}$ corresponding to the maximal edges $\{a, x_1\}$ and $\{x_0, b\}$ respectively ($x_0$ takes value $k$ so it is on the left, $x_1$ take value $-k$ so it is on the right). No other edges touches $R_X$ and $L_X$ so uniqueness of the perfect matching is preserved. Lastly, since $X(F)$ has an unused port, namely $x_2$, it is used for the vertical cover of the new vertex. The horizontal cover of the new $(p+q)$-face can be chosen to be any of the two maximal edges. 

Item $(H1b)$ in Step 2 handles the latter case (only used when $q\geq 4$). Here two new vertices are needed because the two ends of $\cE_i$ lie in different faces. If we write the old faces as $F=(a\ A)$ and $G=(b\ B)$, then the rewiring performed at $(H1b)$ simultaneously enlarges both old faces by $q$ and the 
two new faces are
\[
\begin{aligned}
 &(a\ x_1\ y_1\ y_4\cdots y_{q-1}\ y_0\ x_2\ A),\\
 &(b\ x_3\ y_3\ y_2\ x_4\cdots x_{q-1}\ x_0\ B),
\end{aligned}
\]
where the strings beginning with index $4$ are empty when $q=4$. The new maximal edges are $\{a, x_0\}$, $\{x_2,b\}$, $\{y_0,x_1\}$ and $\{y_3, y_1\}$ which correspond to the $T_{u'}$ auxiliary edges $L_{\orb a} R_X$, $L_X R_{\orb b}$, $L_Y R_X$ and $L_Y R_Y$.
The edges $L_{\orb a} R_X$, $L_X R_{\orb b}$ and $L_Y R_Y$ form a matching in $T_{u'}$ which replaces the old matching edge. The additional
maximal edge $L_Y R_X$ is harmless because $R_Y$ is only incident to the edge $L_Y R_Y$, so $L_YR_X$ cannot be used, and the perfect matching uniqueness property is preserved.  Finally, $(x_3\ y_2)$ is deliberately zero-valued so it can serve as a submaximal edge joining $X$ and $Y$ and is used for the vertical cover of the two new vertices.  

When $\cE_i$ is a port or a charged edge, that is, in cases $(H2)$ or $(H3)$, there is no old matching edge available to subdivide, so the new
vertex must therefore carry its own matching edge. This is precisely the role of $(x_2\ x_4)$ in Step 1. It provides a new maximal edge with both sides in the new vertex $X(F)$ and the edge $L_{X(F)} R_{X(F)}$ is added in $T_{u'}$; this edge is isolated and is hence forced to be in any perfect matching. The edge $\{x_1, x_3\}$ added into $\alpha'$ serves two purposes: it makes sure that $X(F)$ is a $q$-cycle in $\sigma'$ and it takes $u'$-value $0$ so it can serve as in the new vertical cover of $X(F)$. When $\cE_i$ is a port, item $(H2)$ in Step 2 simply rewires the new block $X(F)$ into $F$ and the matching was already taken care of in Step 1. When $\cE_i$ is a charged edge (only used when $q\geq 7)$, the rewiring performed in $(H3)$ makes it so that the dart $e$ (the old $1$-face) and all $X(F)$'s darts except $x_{q-1}$ are absorbed into the enlarged face, while $x_{q-1}$ becomes the
new $1$-face and $\{x_{q-2}, x_{q-1}\}$ is the new charged edge. As in $(H2)$, we have already taken care of the matching of the new vertex as well as its vertical cover. The restriction $q\ge7$ ensures that the darts
$x_{q-2},x_{q-1}$ are distinct from the darts $x_1,x_2,x_3,x_4$.

It is immediate to check that $u'$ is antisymmetric under $\alpha'$ and that its sum over edges of the new faces is $0$. 
It is straightforward to check that transitivity is preserved since the quotient remains connected. 

Lastly, it remains to verify that the new certificate contains a bi-cover. The old vertical cover is unchanged,
because its dart support is disjoint from the old horizontal support and the
old $\sigma$-cycles are unchanged.  The new vertices and enlarged faces are covered as follows:
\begin{itemize}[leftmargin=2.2em]
\item In (H1a), $\{x_0,b\}$ is a matching edge, so it horizontally covers the
enlarged face; the unused port $x_2$ vertically covers the new block vertex $X(F)$.
\item In (H1b), the edge $\{a,x_0\}$ is now a matching edge and it horizontally covers the two enlarged faces corresponding to $F$ and $G$.  The submaximal edge
$\{x_3,y_2\}$ vertically covers the two new vertices $X(F)$ and $X(G)$.
\item In (H2) and (H3), the edge $\{x_2,x_4\}$ is a matching edge, and hence
horizontally covers the enlarged face horizontally. The submaximal edge
$\{x_1,x_3\}$ covers the new vertex $X(F)$.
\end{itemize}
The displayed horizontal and vertical supports are disjoint.  Together with
the unchanged old vertical cover, they give a bi-cover of the new
certificate.
\end{proof}

\subsection{Diagonal sewing}

In the diagonal sewing operation we enlarge every $p$-face and every vertex by one.
The first insertion below (Insertion A) enlarges one face and one vertex simultaneously.
The second (Insertion B) enlarges one vertex without changing any $p$-face; it is
used at precisely the vertices missed by the first insertion. In what follows, we begin with a $(p,q)$-certificate $(D,\varphi,\sigma,\alpha,u)$ which we would like to sew. Whenever a new dart is added to $D$, all old permutations are first extended by fixing the new dart.

\begin{lemma}[Insertion A: enlarge one face and one vertex]\label{lem:typeA}
Let $c\in D$ satisfy $|u(c)|<k$, let $x$ be a new dart, put $D'=D\cup \{x\}$ and $\tau=(c\ x)$, and
define
\[
 \sigma'=\tau\sigma,\qquad
 \alpha'=\tau\alpha\tau,\qquad
 \varphi'=\varphi\tau.
\]
Set $u'(c)=0$, $u'(x)=u(c)$, and leave all other values unchanged.  Then in $(D',\varphi',\sigma',\alpha',u')$ the
face and vertex containing $c$ each gain one dart, and $c$ becomes a port.  If $c$ was already a port, then the new dart $x$ is also a port. If $c$ was not a port, with $e=\alpha c\neq c$, then the old edge $\{c,e\}$ is replaced by the port $c$ and the submaximal edge $\{x,e\}$.
Apart from these two changes, all certificate conditions are
preserved.  In particular, the maximal bipartite graph $T_u$ is unchanged.
\end{lemma}
\begin{proof}
Since $\alpha'$ is conjugate to $\alpha$, it is again an involution.  Moreover,
\[
 \alpha'\sigma'\varphi'
 =\tau\alpha\tau\tau\sigma\varphi\tau
 =\tau(\alpha\sigma\varphi)\tau=1.
\]
The transposition $\tau$ inserts $x$ into the face and the vertex containing
$c$.  If $e=\alpha c\ne c$, then
\[
       \alpha'c=c,\qquad \alpha'x=e,\qquad \alpha'e=x.
\]
Thus the old edge $\{c,e\}$ is replaced by the port $c$ and the edge
$\{x,e\}$.  Moving the value $u(c)$ from $c$ to $x$ preserves antisymmetry and
the sum of $u$ on the enlarged face.  If $c$ was already a port, then $u(c)=0$, and
both $c$ and $x$ are fixed by $\alpha'$ with $u'$-value zero.

Only submaximal values are moved.  Hence no maximal edge
changes, so $T_u'$ equals $T_u$. If $\{c,e\}$ joined two
vertices, it is merely replaced by $\{x,e\}$, with $x$ in the enlargement of
the old vertex containing $c$; transitivity is therefore preserved.
\end{proof}

\begin{lemma}[Insertion B: enlarge one vertex only]\label{lem:typeB}
Let $c\in D$ be a port, let $x$ be a new dart, put $D'=D\cup \{x\}$ and $\tau=(c\ x)$, and define
\[
 \sigma'=\tau\sigma,\qquad
 \alpha'=\tau\alpha,\qquad
 \varphi'=\varphi,
 \qquad u'(x)=0.
\]
Then the vertex containing $c$ gains the dart $x$, $x$ is a $1$-face, and $\{c,x\}$ is a charged edge with both endpoints belonging to the enlarged vertex.
All other certificate conditions are
preserved, and $T_{u'}$ equals $T_u$.
\end{lemma}
\begin{proof}
The extended permutation $\alpha$ fixes both $c$ and $x$, and therefore
commutes with $\tau$.  Thus $\alpha'=\tau\alpha$ is an involution, and
\[
       \alpha'\sigma'\varphi'
       =\tau\alpha\tau\sigma\varphi
       =\alpha\sigma\varphi=1.
\]
The permutation $\varphi$ is unchanged, so $x$ is a $1$-face and none of the old faces are
altered.  The transposition in $\sigma'$ inserts $x$ into the vertex containing
$c$, while $\alpha'$ pairs $c$ with $x$.  Both values are zero.  Hence all face
equations remain valid, the new edge is submaximal, so the auxiliary graph $T_u'=T_u$. 
\end{proof}

The next definition encapsulates the condition that is renewed by one
diagonal step.  The bi-cover is part of the condition because the
output should again be ready for either horizontal or vertical sewing.  Note that unlike the horizontal cover, we do not require that the vertical cover is disjoint from the diagonal data. This is because the diagonal step
may consume the old vertical cover and constructs a new one.

\begin{definition}[Diagonal bi-cover]\label{def:Dframe}
A {\bf diagonal bi-cover}
on a $(p,q)$-certificate consists of the following:
\begin{enumerate}
    \item A bi-cover,
    \item A finite set of submaximal edges with disjoint darts $\{\cE_i\}_{i\in I}$ where $\cE_i$ has two distinct darts $a_i,b_i$ each belonging to a $p$-face and we record their order $(a_i,b_i)$ as well. 
\item A chosen port $z_v$ for every $v\in O_a$ where
\[
 O_a=\mathcal V \setminus \{\orb{a_i}:i\in I\}
\]
is the set of vertices missed by the $a_i$, 
\end{enumerate} 
such that the following conditions hold: 

\begin{enumerate}[(a)]
\item The darts $a_i$  meet every $p$-face exactly once, 
and no two of them belong to the same vertex.  The darts
$b_i$ also satisfy the same two conditions.

\item $O_a$ and $O_b$ are disjoint, where 
\[
        O_b=\mathcal V \setminus \{\orb{b_i}:i\in I\},
\]
is the set of vertices missed by the $b_i$'s. 
Equivalently, every vertex contains a dart from the $a_i$'s or the $b_i$'s or both. 

\item The darts of the edges $\{\cE_i\}_{i\in I}$, the chosen ports $z_v$, and the
horizontal cover are pairwise disjoint (note that we do not require the vertical cover to be disjoint from the $\cE_i$'s and $z_v$'s).
\end{enumerate}
\end{definition}

We are now prepared to define our diagonal sewing operation. 

\begin{definition}\label{def:DSewing} Given a $(p,q)$-certificate $(D,\varphi,\sigma,\alpha,u)$ with a
diagonal bi-cover, the {\bf diagonal sewing} of it is the quintuple $(D',\varphi', \sigma',\alpha', u')$ constructed as follows: \\

\noindent \underline{Step 1}: For each $i\in I$, introduce a new dart $x_i$ and apply insertion A (as in Lemma \ref{lem:typeA}) with $c=a_i$ and $x=x_i$. \\

\noindent \underline{Step 2}: For each vertex $v\in O_a$, introduce a new dart $y_v$ and apply insertion B (as in Lemma \ref{lem:typeB})
with $c=z_v$ and $x=y_v$. 
\end{definition}

All that remains is to verify that the resulting certificate admits a diagonal bi-cover. 

\begin{theorem}[Diagonal sewing]\label{thm:DSewing} Suppose that $(D,\varphi, \sigma,\alpha, u')$ is a $(p,q)$-certificate which admits a diagonal bi-cover. Then the diagonal sewing operation defined in Definition \ref{def:DSewing} produces a $(p+1,q+1)$-certificate $(D',\varphi', \sigma',\alpha', u')$ which admits a diagonal bi-cover.
\end{theorem}
\begin{proof} Write $\cV_a$ for the vertices incident to the darts $\{a_i\}_{i\in I}$ and similarly $\cV_b$ for the vertices incident to $\{b_i\}_{i\in I}$. By Lemma \ref{lem:typeA}, after Step 1 each $p$-face and each vertex in $\cV_a$ have been enlarged by one. By Lemma \ref{lem:typeB}, after Step 2 the remaining vertices, that is, the vertices of $O_a$ are also enlarged by one. Thus, the new faces are of size $1$ and $p+1$ and the new vertices are all of size $q+1$. Lemmas \ref{lem:typeA} and \ref{lem:typeB} guarantee that all other  certificate conditions continue to hold, hence  $(D',\varphi', \sigma',\alpha', u')$ is a $(p+1,q+1)$-certificate.

It remains to show it exhibits a diagonal bi-cover. The horizontal cover was disjoint from all other darts, no new faces of size strictly larger than $1$ were created, and the auxiliary graph $T_{u'}=T_u$ has not changed; therefore we may continue using the old horizontal cover in the new certificate. Next, for the vertical cover, Step 1 makes every $a_i$ a port and these ports lie in distinct enlarged vertices corresponding to the vertices of $\cV_a$. Moreover, for each $v \in O_a$, in Step 2 the submaximal edge $\{z_v, y_v\}$ is created, so together with the $a_i$'s we get a vertical cover that is disjoint from the horizontal one. Note that we have ignored and possibly ran over the old vertical cover. 

Next, by Lemma \ref{lem:typeA} the new edges $\cE_i'=(b_i, x_i)$ for $i\in I$ are submaximal edges, with pairwise disjoint darts that are also disjoint from the horizontal cover's support. The darts $b_i$ meet every enlarged face precisely once, and lie in distinct vertices. The same is true for the $x_i$'s, so condition (a) of Definition \ref{def:Dframe} holds.

After identifying the old vertices
with their enlargements, we have that $\cV_a'=\cV_b$ since the $b_i$'s take the place of the $a_i$'s in the new certificate, and $\cV_b'=\cV_a$ since the $x_i$'s enlarge the vertices incident to the $a_i$'s. So their omitted sets $O_a'=O_b$ and $O_b'=O_a$ are again disjoint. By Lemma \ref{lem:typeA}, the dart $a_i$ in the old certificate has turned into a port in $\cV_a$ and since it has been left unused, we may take it to be the vertex $z_v$ where $v\in O_a'=O_b \subset \cV_a$ is the corresponding enlarged vertex of $a_i$. This shows that that the new certificate contains a diagonal bi-cover and concludes the proof.
\end{proof}

\section{Proof of Theorem \ref{thm:admissible}} \label{sec:existenceproof}


\subsection{Five infinite diagonal strips $p-q\in\{-2,-1,0,1,2\}$}
\begin{theorem}\label{thm:5strips} For any $p,q$ satisfying ${1 \over p}+{1 \over q} \leq 1/2$ and $p-q\in \{-2,-1,0,1,2\}$ there exists a $(p,q)$-certificate with a bi-cover.
\end{theorem}
\begin{proof}
The certificates for the case $p-q=0,-1,-2$, excluding the pairs $(4,5), (4,6), (5,7)$ are given in $(E0)$, $(E1)$ and $(E2)$, respectively, of Proposition \ref{prop:easypairs}. We only need to verify they have a bi-cover. In all of these there is just one vertex and one $p$-face so we may take the unique matched edge for the horizontal cover, and any other submaximal edge for the vertical cover. For the three excluded pairs above, see \eqref{eq:C45},\eqref{eq:C46} and \eqref{eq:C57} in Section \ref{app:seeds} for explicit certificates with a bi-cover. 

For the case $p-q=1$, in \eqref{eq:C54} in Section \ref{app:seeds} we provide a $(5,4)$-certificate with a diagonal bi-cover. Therefore we may apply Theorem \ref{thm:DSewing} successively and obtain a $(5+r,4+r)$-certificate for any integer $r\geq 0$. Similarly, for the case $p-q=2$, in \eqref{eq:C64} in Section \ref{app:seeds} we provide a $(6,4)$-certificate with a diagonal bi-cover, so by Theorem \ref{thm:DSewing} we obtain $(6+r,4+r)$-certificates for any integer $r\geq 0$. 
\end{proof}

\subsection{Pairs $(3,q)$ with $q \geq 8$ and $(p,3)$ with $p\geq 8$}

\begin{theorem}\label{thm:3qp3} For every integers $p,q\geq 8$ there exist a $(3,q)$-certificate and a $(p,3)$-certificate both with bi-covers. 
\end{theorem}

For pairs $(3,q)$ with $q\geq 8$ we would like to start with the certificate for the Euclidean pair $(3,6)$ (which is obtained just by folding onto a torus with $2$ faces, see below) and use vertical sewing (as in Definition \ref{def:VSewing}). However, this only gives us pairs of the form $(3,6+3r)$ for $r \geq 0$ and will not give us all pairs $(3,q)$ with $q \geq 8$. Instead we use a very similar, yet ad-hoc vertical sewing procedure. 

We begin with a $(3,6)$-certificate. This is a folding of the $6$-regular triangulation into a torus. It can be visualized by taking a parallelogram obtained by gluing two triangles along an edge, then identifying top to bottom and left to right edges to obtain a torus with one vertex, three edges and two faces. Formally, on $D=\{0,\dots,5\}$ let
\begin{equation}\label{eq:C36}
\begin{gathered}
 \sigma=(0\ 1\ 2\ 3\ 4\ 5),\qquad
 \varphi=(0\ 2\ 4)(1\ 3\ 5),\\
 \alpha=(0\ 3)(1\ 4)(2\ 5),\qquad
 u=(-2,-1,1,2,1,-1) \, .
\end{gathered}
\end{equation}
It is immediate to check that all the conditions in Definition \ref{def:admissible_permutations}. Note that here we have a single vertex, $\max u =2$, the edge $\{3,0\}$ is the only matching edge, and that there is a submaximal edge (taking values $\pm 1$). In the following successive operations we will make sure these properties are preserved. 

Let $(a,b)$ be an edge with $u$-values $1$ and $-1$ respectively; initially, in the $(3,6)$-certificate above we take $(a,b)=(1, 4)$. We make one of the two insertions:
\begin{itemize}
    \item[(S3)] Add three new darts $D' = D \cup \{x_0,x_1,x_2\}$, take $\varphi'$ by setting it to be $\varphi$ on $D$ and a cycle $(x_0\ x_1\ x_2)$ on $D'\setminus D$, initialize $\alpha'$ be setting it to be $\alpha$ on $D$ and fixing $\{x_0,x_1,x_2\}$, then replace    \begin{equation}\label{eq:S3}
       (a\ b)(x_0)(x_1)(x_2)
       \quad\text{by}\quad
       (a\ x_0)(b\ x_1)(x_2),
    \end{equation}
in $\alpha'$ and set $\sigma'= \alpha' (\varphi')^{-1}$. Finally set $u':D'\to \Z$ by taking the values of $u$ on $D$ and giving $x_0,x_1,x_2$ the values $1, -1, 0$, respectively. It is straightforward to check that this increased the size of the unique vertex by $3$ and that  $(D',\varphi', \sigma',\alpha', u')$ is a $(3,q+3)$-certificate. The edge $(0, 3)$ remains the unique maximal edge and the new edge $(a,x_0)$ takes the values $1$ and $-1$. 

\item[(S4)] Do the same as in $(S3)$ and add an additional dart $y$ which will be a $1$-face, so that $\varphi'$ adds a $3$-face $(x_0\ x_1\ x_2)$ and a $1$-face $(y)$ to $\varphi$. Replace 
\begin{equation}\label{eq:S4}
       (a\ b)(x_0)(x_1)(x_2)(y)
       \quad\text{by}\quad
       (a\ x_0)(b\ x_1)(x_2\ y),
\end{equation}
in $\alpha'$ and set $\sigma'= \alpha' (\varphi')^{-1}$. Finally, set $u'$ by giving values $1,-1,0,0$ to $x_0,x_1,x_2,y$, respectively. The vertex was enlarged by $4$ and and we obtain a $(3,q+4)$-certificate. Again, the edge $(0, 3)$ remains the unique maximal edge and the edge $(a,x_0)$ takes the values $1$ and $-1$. 
\end{itemize}

\begin{lemma}\label{lem:p3semigroup}
For all integers $r,s\geq 0$ there exists a $(3,6+3r+4s)$-certificate with a bi-cover. 
\end{lemma}
\begin{proof} 
By applying $r$ times the operation $(S3)$ and $s$ times the operation $(S4)$, by the discussion above we obtain a $(3,6+3r+4s)$-certificate. So it only remains to verify it has a bi-cover. 
 
For the horizontal cover, the matching edge $(0,3)$ covers the two original faces $(0\ 2\ 4)$ and $(1\ 3\ 5)$; each new triangle added in $(S3)$ has a port $x_2$ and each triangle added in $(S4)$ has a charged edge $(x_2, y)$. The disjoint vertical cover is provided by the edge $(a,b)$ which we made sure is submaximal. 
\end{proof}

\begin{proof}[Proof of Theorem \ref{thm:3qp3}]
For the pairs $(3,q)$ with $q \geq 8$,
by Lemma \ref{lem:p3semigroup} above we immediately get the desired assertion for $q=6,9,10,12,13,14\ldots$ since $\{3r+4s:r,s\geq 0\}$ are all the naturals except $\{1,2,5\}$. In \eqref{eq:C38} of Section \ref{app:seeds} we present an explicit $(3,8)$-certificate with a bi-cover, and a $(3,11)$-certificate with a bi-cover is obtained from it using Theorem \ref{thm:V}.    

For the pairs $(p,3)$ with $p\geq 8$ we start with the explicit $(8,3)$ and $(10,3)$ certificates with bi-covers presented in \eqref{eq:C83} and \eqref{eq:C103} in Section \ref{app:seeds} as well as the following Euclidean $(6,3)$-certificate with bi-cover. On $D=\{0,\ldots,5\}$ set
\begin{equation}\label{eq:C63}
\begin{gathered}
 \sigma=(0\,1\,2)(3\,4\,5),\qquad
 \varphi=(0\,5\,1\,3\,2\,4),\\
 \alpha=(0\ 3)(1\ 4)(2\ 5),\qquad
 u=(1,0,-1,-1,0,1) \, ,
\end{gathered}
\end{equation}
and cover it horizontally with the matching edge $\{0,3\}$ and vertically with the submaximal edge $\{1,4\}$. These three certificates together with the horizontal sewing of Theorem \ref{thm:H} yields all pairs $(p,3)$ with $p \geq 8$.
\end{proof}

\subsection{Proof of Theorem \ref{thm:admissible}} Let $(p,q)$ satisfy 
\begin{equation}\label{eq:pq-condition}
{1 \over p} + {1\over q} \leq {1 \over 2} \, .
\end{equation} Subtract the smaller coordinate from the larger one as long as the result still has coordinates satisfying \eqref{eq:pq-condition}. The process clearly terminates since one coordinate decreases at every step. Denote the terminal pair by $(p',q')$. If $\{|p'-q'|\leq 2\}$, that is, it is one of the five infinite diagonal strips, then by Theorem \ref{thm:5strips} and tracing the steps back to $(p,q)$ using Theorems \ref{thm:V} and \ref{thm:H} we get that there exists a $(p,q)$-certificate. 

Otherwise, $|p'-q'|> 2$. Assume further that $p'>q'$ and put $d=p'-q'>2$. Since one more step would have violated \eqref{eq:pq-condition} we get that 
$$ {1 \over d} + {1 \over q'} > {1 \over 2} \, .$$
If $d=3$, then $q'$ can only take the values $3,4$ or $5$; if $d=4,5$, then $q'$ must take the value $3$. In this case $(p',q')$ is one of the five:
$$(6,3)\ (7,4)\ (8,5)\  (7,3)\ (8,3) \, .$$
By an identical reason, if $q' > p'$, then $(p',q')$ is one of the five:
$$ (3,6)\ (4,7)\ (5,8)\ (3,7)\ (3,8)  \, .$$
All of these ten pairs, except $(7,3)$ and $(3,7)$ have a certificate with a bi-cover. Indeed $(8,3)$ and $(3,8)$ are handled by Theorem \ref{thm:3qp3}, the Euclidean pairs $(3,6)$ and $(6,3)$ are handled in \eqref{eq:C36} and \eqref{eq:C63}, and for the other four we provide explicit certificates with bi-covers in \eqref{eq:C74}, \eqref{eq:C47}, \eqref{eq:C85}, \eqref{eq:C58} in Section \ref{app:seeds}. Thus, if $(p',q')$ is one of these eight, we may use Theorems \ref{thm:V} and \ref{thm:H} to conclude that there exists a $(p,q)$-certificate. 

We handle the cases $(3,7)$ and $(7,3)$ differently since we were not able to find certificates with bi-covers in these cases, only certificates. If $(p',q')=(3,7)$ then either $(p,q)=(3,7)$ or the step before was $(3,10)$ or $(10,7)$. In the former case, we provide an explicit $(3,7)$-certificate in \eqref{eq:C37}. In the latter case $(3,10)$ is handled by Theorem \ref{thm:3qp3} and $(10,7)$ by the explicit certificate in \eqref{eq:C107} in Section \ref{app:seeds}. Similarly, if $(p',q')=(7,3)$ we conclude using \eqref{eq:C710} and \eqref{eq:C73} of Section \ref{app:seeds} and Theorem \ref{thm:3qp3}. This concludes the proof of the theorem. \qed


\section{Explicit certificates}\label{app:seeds}

\paragraph{A $(4,5)$-certificate with a bi-cover.} On $D=\{0,\dots,19\}$ set

\begin{equation}
\begin{aligned}
\sigma&=(0\,1\,2\,3\,4)(5\,6\,7\,8\,9)(10\,11\,12\,13\,14)(15\,16\,17\,18\,19),\\[2pt]
\alpha&=(0\,8)(1\,15)(2\,16)(3\,7)(4)(5\,11)(6\,9)(10\,12)(13\,17)(14\,19)(18),\\[2pt]
\varphi&=(0\,7\,2\,15)(1\,19\,13\,16)(3\,6\,8\,4)(5\,10\,11\,9)(12\,14\,18\,17) ,\\ \\[2pt]
u&=(0,-1,0,1,0,1,-1,-1,0,1,-1,-1, 1,1,0,1,0,-1,0,0) \, .\\ 
\end{aligned}
\label{eq:C45}
\end{equation}
Here $\max u=1$; the maximal matching edges are \[(13\ 17),\,(5\ 11),\,(3\ 7),\,(15\ 1) \, .\]
The horizontal cover are the first three matching edges above and the
disjoint vertical cover are the submaximal edges $(0\ 8)$ and $(14\ 19)$. 


\paragraph{A $(4,6)$-certificate with a bi-cover.} On $D=\{0,\dots,11\}$ set

\begin{equation}
\begin{aligned}
\sigma&=(0\,5\,4\,3\,2\,1)(6\,11\,10\,9\,8\,7),\\[2pt]
\alpha&=(0\,4)(1)(2\,9)(3\,10)(5\,7)(6\,8)(11),\\[2pt]
\varphi&=(0\,5\,8\,7)(1\,2\,10\,4)(3\,11\,6\,9),\\ \\[2pt]
u&=(1,0,1,0,-1,1,1,-1,-1,-1,0,0) \, . \\
\end{aligned}
\label{eq:C46}
\end{equation}
Here $\max u=1$; the maximal matching edges are 
\[(0\ 4),\,(6\ 8).\]
The horizontal cover is the matching edge $(6,8)$ and the port $(1)$. 
The disjoint vertical data is the submaximal edge $(3\ 10)$.

\paragraph{A $(4,7)$-certificate with a bi-cover.} On $D=\{0,\dots,13\}$ set

\begin{equation}
\begin{aligned}
\sigma&=(0\,6\,8\,9\,5\,12\,4)(1\,7\,3\,10\,11\,2\,13),\\[2pt]
\alpha&=(0\,7)(1\,13)(2\,10)(3\,6)(4\,12)(5\,8)(9\,11),\\[2pt]
\varphi&=(0\,1\,2\,3)(4\,5\,6\,7)(8\,9\,10\,11)(12)(13),\\ \\[2pt]
u&=(2,0,-1,-1,0,1,1,-2,-1,-2,1,2, 0,0).
\end{aligned}
\label{eq:C47}
\end{equation}
Here $\max u=2$; the matching edges are $(0,7)$ and $(9,11)$. The vertical cover is the submaximal edge $(3,6)$ and the disjoint horizontal cover is the matching edge $(9,11)$ and the two charged edges $(1,13)$ and $(4,12)$.


\paragraph{A $(5,7)$-certificate with a bi-cover.} On $D=\{0,\dots,20\}$ set
\begin{equation}
\begin{aligned}
\sigma&=(0\,10\,3\,9\,5\,2\,1)(4\,14\,13\,12\,15\,16\,11)(6\,8\,17\,19\,18\,7\,20),\\[2pt]
\alpha&=(0)(1)(2\,9)(3\,14)(4\,10)(5\,8)(6\,20)(7\,17)(11\,15)(12)(13)(16\,19)(18),\\[2pt]
\varphi&=(0\,1\,2\,3\,4)(5\,6\,7\,8\,9)(10\,11\,12\,13\,14)(15\,16\,17\,18\,19)(20),\\ \\[2pt]
u&=(0,0,-1,1,0,1,0,-1,-1,1,0,1,0,0,-1,-1,1,1,0,-1,0) \, .
\end{aligned}
\label{eq:C57}
\end{equation}
Here $\max u=1$; the maximal matching edges are 
\[(9\ 2)\, (11\ 15) \,(17\  7).\]
The horizontal cover is the matching edge $(11\ 15)$, the charged edge $(6\ 20)$ and the port $0$. The disjoint vertical cover are the ports $(1), (12)$ and $(18)$. 

\paragraph{A $(5,8)$-certificate with a bi-cover.} On $D=\{0,\dots,15\}$ set

\begin{equation}
\begin{aligned}
\sigma&=(0\,7\,14\,15\,13\,8\,4\,3)(1\,2\,9\,12\,11\,10\,6\,5),\\[2pt]
\alpha&=(0\,2)(1\,9)(3)(4\,7)(5)(6\,14)(8\,12)(10)(11)(13\,15),\\[2pt]
\varphi&=(0\,1\,2\,3\,4)(5\,6\,7\,8\,9)(10\,11\,12\,13\,14)(15),\\ \\[2pt]
u&=(1,1,-1,0,-1,0,0,1,0,-1,0,0,0,0,0,0)
\end{aligned}
\label{eq:C58}
\end{equation}
Here $\max u=1$; the matching edges are $(7,4)$ and $(1,9)$ (there is another maximal edge $(0,2)$ but it cannot be used in a perfect matching). The horizontal cover are the ports $3,5$ and $10$ and the disjoint vertical cover is the submaximal edge $\{6,14\}$. 

\paragraph{A $(7,10)$-certificate with a bi-cover.} On  $D=\{0,\dots,9\}$ set
\begin{equation}
\begin{aligned}
\sigma&=(0\,1\,2\,3\,4\,5\,6\,7\,8\,9),\\[2pt]
\alpha&=(0\,1)(2\,3)(4\,5)(6\,8)(7\,9),\\[2pt]
\varphi&=(0)(1\,9\,6\,7\,8\,5\,3)(2)(4) ,\\ \\[2pt]
u&=(0,0,0,0,0,0,1,0,-1,0)
\end{aligned}
\label{eq:C710}
\end{equation}
Here $\max u=1$; the matching edge is $(6,8)$. The horizontal cover is the charged edge $(1,0)$ andthe disjoint vertical cover is the submaximal edge $(2,3)$. 

\paragraph{A $(5,4)$-certificate with a diagonal bi-cover.}
On $D=\{0,\dots,19\}$ set

\begin{equation}
\begin{aligned}
\sigma&=(0\,15\,2\,7)(1\,16\,13\,19)(3\,4\,8\,6)(5\,9\,11\,10)(12\,17\,18\,14),\\[2pt]
\alpha&=(0\,8)(1\,15)(2\,16)(3\,7)(4)(5\,11)(6\,9)(10\,12)(13\,17)(14\,19)(18),\\[2pt]
\varphi&=(0\,4\,3\,2\,1)(5\,9\,8\,7\,6)(10\,14\,13\,12\,11)(15\,19\,18\,17\,16),\\ \\[2pt]
u&=(2,1,-2,-1,0,1,2,1,-2,-2,2,-1, -2,-1,2,-1,2,1,0,-2)\, . \\
\end{aligned}
\label{eq:C54}
\end{equation}
Here $\max u=2$; the matching edges are
\[(0\ 8),\, (2\ 16), \,(6\ 9) ,\, (10\ 12) ,\, (14\ 19) .\]
The horizontal cover are the matching edges $(0\ 8)$ and $(14\ 19)$. The vertical cover are the submaximal edges $(1\ 15)$ and $(5\ 11)$, and the ports $(4)$ and $(18)$. The extra diagonal data needed to satisfy Definition \ref{def:Dframe} are the two-element submaximal  edges (viewed as ordered pairs)
$$\{E_i\}_{i=1}^4 = \big \{ (1\ 15),\, (7\ 3),\, (11\ 5),\, (17\ 13) \big \}$$
so that $O_a$ contains a single vertex $(3\,4\,8\,6)$ in which we take the port $z_v=(4)$; since $b_2=3$ the unique vertex of $O_a$ belongs to $V_b$, so $O_b$ is disjoint from $O_a$. Lastly, all of these last darts are disjoint from the horizontal cover.

\paragraph{A $(6,4)$-certificate with a diagonal bi-cover.}
On $D=\{0,\dots,11\}$ set

\begin{equation}
\begin{aligned}
\sigma&=(0\,7\,8\,5)(1\,4\,10\,2)(3\,9\,6\,11),\\[2pt]
\alpha&=(0\,4)(1)(2\,9)(3\,10)(5\,7)(6\,8)(11),\\[2pt]
\varphi&=(0\,1\,2\,3\,4\,5)(6\,7\,8\,9\,10\,11),\\[2pt] \\
u&=(2,0,2,-1,-2,-1,2,1,-2,-2,1,0)\, . \\
\end{aligned}
\label{eq:C64}
\end{equation}
Here $\max u=2$; the maximal matching edges are 
\[(0\ 4), \,(2\ 9), \, (6\ 8) \, .\]
The horizontal cover are the matching edges $(0\ 4)$ and $(6\ 8)$ (we could have also taken just $(2\ 9)$) and the disjoint vertical cover is the submaximal edge $(5\ 7)$ and the ports $(1)$ and $(11)$. The extra diagonal data needed to satisfy Definition \ref{def:Dframe} are the 
the two-element submaximal edges (viewed as ordered pairs)
$$\{E_i\}_{i=1}^2 = \big \{ (3\ 10),\, (7\ 5) \big \}$$
so that $O_a$ contains a single vertex $(1\,4\,10\, 2)$ in which we take the port $z_v=(1)$; since $b_1=10$ the unique vertex of $O_a$ belongs to $V_b$, so $O_b$ is disjoint from $O_a$. Lastly, all of these last darts are disjoint from the horizontal cover. 

\paragraph{A $(7,4)$-certificate with a bi-cover.}
On $D=\{0,\dots,7\}$ set
\begin{equation}
\begin{aligned}
\sigma&=(0\,1\,4\,5)(2\,6\,3\,7),\\[2pt]
\alpha&=(0\,4)(1\,6)(2\,7)(3\,5),\\[2pt]
\varphi&=(0\,1\,2\,3\,4\,5\,6)(7) ,\\[2pt] \\
u&=(0,1,0,1,0,-1,-1,0)
\end{aligned}
\label{eq:C74}
\end{equation}
Here $\max u=1$; the matching edges are $(1,6)$ and $(3,5)$. 
The horizontal cover is the edge $(1,6)$ and the disjoint vertical cover are the submaximal edges $(0,4)$ and $(2,7)$.

\paragraph{A $(10,7)$-certificate with a bi-cover.}
On $D=\{0,\dots,13\}$ set
\begin{equation}
\begin{aligned}
\sigma&=(0\,10\,9\,5\,7\,1\,13)(2\,6\,8\,4\,11\,3\,12),\\[2pt]
\alpha&=(0\,13)(1\,6)(2\,12)(3\,11)(4\,7)(5\,8)(9\,10),\\[2pt]
\varphi&=(0\,1\,2\,3\,4\,5\,6\,7\,8\,9)(10)(11)(12)(13) ,\\[2pt] \\
u&=(0,1,0,0,1,0,-1,-1,0,0,0,0,0,0) \, .
\end{aligned}
\label{eq:C107}
\end{equation}
Here $\max u=1$; the matching edges are $(1,6)$ and $(4,7)$. The horizontal cover is the charged edge $(0,13)$ and the disjoint vertical cover is the submaximal edge $(5,8)$.


\paragraph{A $(8,5)$-certificate with a bi-cover.}
On $D=\{0,\dots,9\}$ set

\begin{equation}
\begin{aligned}
\sigma&=(0\,2\,6\,8\,5)(1\,4\,9\,3\,7),\\[2pt]
\alpha&=(0\,4)(1\,6)(2\,7)(3\,9)(5\,8),\\[2pt]
\varphi&=(0\,1\,2\,3\,4\,5\,6\,7)(8)(9)  ,\\[2pt] \\
u&=(1,1,0,0,-1,0,-1,0,0,0)\,.
\end{aligned}
\label{eq:C85}
\end{equation}
Here $\max u=1$; the matching edges are $(0,4)$ and $(1,6)$.
The horizontal cover is the matching edge $(0,4)$ and the disjoint vertical cover is the submaximal edge $(2,7)$.

\paragraph{A $(3,8)$-certificate with a bi-cover.} On $D=\{0,\dots,23\}$ set

\begin{equation}
\begin{aligned}
\sigma&=
(0\,1\,2\,3\,4\,5\,6\,7)(8\,9\,10\,11\,12\,13\,14\,15)
(16\,17\,18\,19\,20\,21\,22\,23),\\[2pt]
\alpha&=(0)(1\,8)(2\,5)(3\,16)(4\,17)(6\,15)(7\,9)
(10\,14)(11\,19)(12\,22)(13\,20)(18\,23)(21)
,\\[2pt]
\varphi&=(0\,7\, 8)(1\,15\,5)(2\,4\,16)(3\,23\,17)(6\,14\,9)(10\,13\,19)(11\,18\,22)(12\,21\,20) ,\\[2pt] \\
u&=(0,1,-1,-2,-1,1,2,1,-1,-1,1,2,1,1,-1,-2,2,1,-1,-2,-1,0,-1,1) \, .\\
\end{aligned}
\label{eq:C38}
\end{equation}
Here $\max u=2$; the maximal matching edges are 
\[(6\ 15),\, (11\ 19),\,(16\ 3).\]
The horizontal cover the three matching edges above (covering six faces) and the ports $(0)$ and $(21)$ (covering the remaining two). The disjoint vertical are the submaximal edge $(1\ 8)$ and $(18,23)$.


\paragraph{A $(8,3)$-certificate with a bi-cover.} On $D=\{0,\dots,23\}$ set

\begin{equation}
\begin{aligned}
\sigma&=(0\,2\,1)(3\,5\,4)(6\,8\,7)(9\,11\,10)(12\,14\,13)(15\,17\,16)(18\,20\,19)
(21\,23\,22),\\[2pt]
\alpha&=(0\,10)(1)(2\,5)(3\,14)(4\,9)(6\,17)(7\,22)(8\,20)(11\,16)(12\,18)(13\,19)(15\,23)(21),\\[2pt]
\varphi&=(0\,11\,17\,7\,23\,16\,9\,5)(1\,2\,3\,12\,19\,14\,4\,10)(6\,15\,21\,22\,8\,18\,13\,20) ,\\ \\[2pt]
u&=(2,0,-2,-2,1,2,-1,-2,2,-1,-2,2,1,-2,2,2,-2,1,-1,2,-2,0,2,-2) \, .\\
\end{aligned}
\label{eq:C83}
\end{equation}
Here $\max u=2$; the maximal matching edges are\[(0\ 10),\,(5\ 2),\,(8\ 20),\,(11\ 16),\,(14\ 3),\,(15\ 23),\,(19\ 13),\,(22\ 7).\]
The horizontal cover are the matching edges $(11\ 16), (14\ 3)$ and $(8\ 20)$ and the disjoint vertical data are the ports $(1), (21)$ and the submaximal edges $\{6,17\}, \{12,18\}$ and $\{4,9\}$.

\paragraph{A $(10,3)$-certificate with a bi-cover.} On $D=\{0,\dots,29\}$ set

\begin{equation}
\begin{aligned}
\sigma&=(0\,2\,1)(3\,5\,4)(6\,8\,7)(9\,11\,10)(12\,14\,13)(15\,17\,16)(18\,20\,19) (21\,23\,22)(24\,26\,25)(27\,29\,28),\\[2pt]
\alpha&=(0\,15)(1\,7)(2\,4)(3\,13)(5\,16)(6\,18)(8\,26)(9\,14)(10\,29)(11\,25)(12\,21)\\
&\quad \,\, (17)(19\,24)(20)(22\,27)(23)(28),\\[2pt]
\varphi&=(0\,16\,3\,14\,10\,27\,23\,21\,13\,4)(1\,8\,24\,20\,18\,7\,2\,5\,17\,15) (6\,19\,25\,9\,12\,22\,28\,29\,11\,26),\\ \\[2pt]
u&=(2,-2,-1,2,1,-2,-2,2,1,-1,-2,2,2,-2,1,-2,2,0,2,-2,0,-2,2,0,\\
&\quad \,\,\,\, 2,-2,-1,-2,0,2).
\end{aligned}
\label{eq:C103}
\end{equation}
Here $\max u=2$; the maximal matching edges are
\[(0\ 15),\,(3\ 13),\,(7\ 1),\,(11\ 25),\,(12\ 21),\,(16\ 5),\,(18\ 6),\,(22\ 27),\,(24\ 19),\,(29\ 10).\]
The horizontal cover are the matching edges $(3\ 13),(7\ 1)$ and $(11\ 25)$, and the disjoint vertical cover are the submaximal edges $(2\ 4),(8\ 26), (9\ 14)$ and the ports $(17), (20), (23), (28)$. 

\paragraph{A $(3,7)$-certificate} On $D=\{1,\bar 1,2,\bar 2,\ldots,20,\overline{20},h_1,h_2\}$ set
\[
\alpha=(1\,\bar1)(2\,\bar2)\cdots(20\,\overline{20}),
\]
and
\begin{equation}
\begin{aligned}
\sigma&=
(1\,\bar2\,\bar3\,4\,5\,6\,3) \,
(7\,\overline{10}\,h_1\,\overline{11}\,\bar8\,12\,9)
\,
(10\,15\,13\,\bar5\,\overline{16}\,14\,11)
\,
(2\,17\,\overline{12}\,\overline{19}\,\bar9\,
 \overline{17}\,\bar1)
\,
(16\,\bar4\,\bar6\,\overline{13}\,20\,h_2\,18)
\\
&\,\, \quad
(8\,\overline{14}\,\overline{18}\,\overline{20}\,
\overline{15}\,\bar7\,19) \, ,
\end{aligned}
\label{eq:C37}
\end{equation}
and
\[
\begin{aligned}
\varphi&=(1\,\overline{17}\,2)
(\bar2\,\bar1\,3)
(\bar3\,6\,\bar4)
(4\,16\,\bar5)
(5\,13\,\bar6)
(7\,\overline{15}\,10)
(h_1\,\overline{10}\,11)
(\bar8\,19\,\overline{12})
(9\,\overline{19}\,\bar7)
(\overline{11}\,14\,8)
(12\,17\,\bar9) 
\\ &\,\,\quad
(15\,\overline{20}\,\overline{13})
(\overline{16}\,18\,\overline{14})
(h_2\,20\,\overline{18}) \, ,
\end{aligned}
\]
and lastly
\[
u(i)=
\begin{cases}
1,
&i\in
\{2,3,5,6,8,9,10,11,15,16,17,18,19,20\},\\
0,
&i\in
\{1,4,7,12,13,14\} \, ,
\end{cases}
\]
and we are forced to set $u(\bar i)=-u(i)$ and $u(h_1)=u(h_2)=0$. The following table specifies which maximal edges connect $L_i$ to $R_j$ in the auxiliary graph $T_u$; since it is lower diagonal, and each entry in the diagonal has a single edge, this forces the perfect matching to be given by the edges on the diagonal.
\[
\begin{array}{c|cccccc}
 & R_4 & R_0 & R_2 & R_5 & R_1 & R_3 \\ \hline
L_0 & 6   & 3   & 5      & 0      & 0  & 0 \\
L_3 & 0   & 2   & 0      & 0      & 0  & 17 \\
L_4 & 0   & 0   & 16     & 18,20      & 0  & 0 \\
L_2 & 0   & 0   & 0  & 15     & 0  & 10,11 \\
L_5 & 0   & 0   & 0      & 0  & 8  & 19 \\
L_1 & 0   & 0  & 0      & 0      & 0 & 9
\end{array}
\]

\paragraph{A $(7,3)$-certificate} On $D=\{1,\bar 1,2,\bar 2,\ldots,20,\overline{20},h_1,h_2\}$ set
\[
\alpha=(1\,\bar1)(2\,\bar2)\cdots(20\,\overline{20}),
\]
and

\begin{equation}
\begin{aligned}
\sigma&=
(h_1\,1\,\bar2)
(\bar1\,\bar5\,3)
(2\,7\,\bar4)
(5\,\overline{10}\,9)
(\bar6\,\bar3\,11)
(\bar7\,13\,\overline{11})
(8\,4\,\overline{14})
(10\,\bar8\,\overline{15})
(\overline{12}\,6\,16)
(\overline{13}\,12\,17)
(14\,\bar9\,\overline{18})
\\
&\,\,\quad
(\overline{17}\,15\,\overline{19})
(18\,\overline{16}\,20)
(19\,h_2\,\overline{20}).
\end{aligned}
\label{eq:C73}
\end{equation}
and
\[
\begin{aligned}
\varphi&=
(h_1\,\bar2\,\bar4\,8\,10\,5\,\bar1)
(1\,3\,\bar6\,\overline{12}\,\overline{13}\,\bar7\,2)
(\bar5\,9\,14\,4\,7\,\overline{11}\,\bar3)
(\overline{10}\,\overline{15}\,\overline{17}\,
  12\,16\,18\,\bar9)
\\
&\,\,\quad
(11\,13\,17\,\overline{19}\,\overline{20}\,
  \overline{16}\,6)
(\overline{14}\,\overline{18}\,20\,h_2\,
  19\,15\,\bar8).
\end{aligned}
\]
and lastly
\[
u(i)=
\begin{cases}
1,
&i\in
\{1,2,3,4,5,6,8,9,10,11,12,13,14,15,16,18,19,20\},\\
0,
&i\in
\{7,17\} \, ,
\end{cases}
\]
and we are forced to set $u(\bar i)=-u(i)$ and $u(h_1)=u(h_2)=0$. As before, the following table of maximal edges shows that $T_u$ has a unique perfect matching:
\[
\begin{array}{c|cccccccccccccc}
 & R_1 & R_{10} & R_{13} & R_4 & R_{12} & R_7 & R_2
 & R_0 & R_5 & R_9 & R_3 & R_8 & R_6 & R_{11} \\ \hline
L_0    & 1 & 0  & 0  & 0  & 0  & 0  & 0 & 0 & 0 & 0 & 0 & 0 & 0 & 0 \\
L_3    & 5 & 9  & 0  & 0  & 0  & 0  & 0 & 0 & 0 & 0 & 0 & 0 & 0 & 0 \\
L_{12} & 0 & 18 & 20 & 0  & 0  & 0  & 0 & 0 & 0 & 0 & 0 & 0 & 0 & 0 \\
L_1    & 0 & 0  & 0  & 3  & 0  & 0  & 0 & 0 & 0 & 0 & 0 & 0 & 0 & 0 \\
L_8    & 0 & 0  & 0  & 6  & 16 & 0  & 0 & 0 & 0 & 0 & 0 & 0 & 0 & 0 \\
L_{11} & 0 & 0  & 0  & 0  & 0  & 15 & 0 & 0 & 0 & 0 & 0 & 0 & 0 & 0 \\
L_6    & 0 & 0  & 0  & 0  & 0  & 8  & 4 & 0 & 0 & 0 & 0 & 0 & 0 & 0 \\
L_2    & 0 & 0  & 0  & 0  & 0  & 0  & 0 & 2 & 0 & 0 & 0 & 0 & 0 & 0 \\
L_4    & 0 & 0  & 0  & 0  & 0  & 0  & 0 & 0 & 11& 0 & 0 & 0 & 0 & 0 \\
L_5    & 0 & 0  & 0  & 0  & 0  & 0  & 0 & 0 & 0 & 13& 0 & 0 & 0 & 0 \\
L_7    & 0 & 0  & 0  & 0  & 0  & 0  & 0 & 0 & 0 & 0 & 10& 0 & 0 & 0 \\
L_9    & 0 & 0  & 0  & 0  & 0  & 0  & 0 & 0 & 0 & 0 & 0 & 12& 0 & 0 \\
L_{10} & 0 & 0  & 0  & 0  & 0  & 0  & 0 & 0 & 0 & 0 & 0 & 0 & 14& 0 \\
L_{13} & 0 & 0  & 0  & 0  & 0  & 0  & 0 & 0 & 0 & 0 & 0 & 0 & 0 & 19
\end{array}
\]

\section*{Acknowledgements} Part of this work was conducted while the author was visiting the Simons Institute for the Theory of Computing. This research was also supported by ERC consolidator grant 101001124 (UniversalMap) and ISF grant 898/23.

\bibliographystyle{abbrv}
\bibliography{tilings}




\end{document}